\documentclass[12pt, reqno]{amsart}
\usepackage{graphicx} 
\usepackage{color}
\usepackage[textsize=small]{todonotes}

\numberwithin{equation}{section}

\newtheorem{definition}{Definition}
\newtheorem{lemma}{Lemma}[section]
\newtheorem{remark}{Remark}[section]
\newtheorem{example}{Example}
\newtheorem{theorem}{Theorem}

\newtheorem{assumption}{Assumption}

\begin{document}

\baselineskip 18pt

\newcommand{\E}[1]{\mathbb{E}\left[ #1 \right]}
\newcommand{\Eof}[1]{\mathbb{E}\left[ #1 \right]}
\renewcommand{\H}{\mathbb{H}}
\newcommand{\R}{\mathbb{R}}
\newcommand{\sigl}{\sigma_L}
\newcommand{\BS}{\rm BS}
\newcommand{\p}{\partial}
\renewcommand{\P}{\mathbb{P}}
\newcommand{\var}{{\rm var}}
\newcommand{\cov}{{\rm cov}}
\newcommand{\beaa}{\begin{eqnarray*}}
\newcommand{\eeaa}{\end{eqnarray*}}
\newcommand{\bea}{\begin{eqnarray}}
\newcommand{\eea}{\end{eqnarray}}
\newcommand{\ben}{\begin{enumerate}}
\newcommand{\een}{\end{enumerate}}
\newcommand{\bit}{\begin{itemize}}
\newcommand{\eit}{\end{itemize}}
\newcommand{\inn}[2]{\langle #1, #2\rangle}
\newcommand{\intl}[1]{\langle #1 \rangle}

\newcommand{\bX}{\boldsymbol{X}}
\newcommand{\bx}{\boldsymbol{x}}
\newcommand{\bE}{\mathbf{e}}
\newcommand{\bw}{\mathbf{w}}
\newcommand{\bW}{\boldsymbol{W}}
\newcommand{\bB}{\boldsymbol{B}}
\newcommand{\bZ}{\boldsymbol{Z}}
\newcommand{\bH}{\mathbf{H}}
\newcommand{\bF}{\mathbf{F}}
\newcommand{\bG}{\mathbf{G}}
\newcommand{\bs}{\mathbf{s}}
\newcommand{\bsihat}{\widehat{\bs_i}}

\newcommand{\cL}{\mathcal{L}}

\newcommand{\mt}{\mathbf{t}}
\newcommand{\mS}{\mathbb{S}}

\newcommand{\argmax}{{\rm argmax}}
\newcommand{\argmin}{{\rm argmin}}

\newcommand{\tM}{\widetilde{M}}
\newcommand{\tE}[1]{\tilde{\mathbb{E}}\left[ #1 \right]}
\newcommand{\tEof}[1]{\tilde{\mathbb{E}}\left[ #1 \right]}
\newcommand{\tP}{\tilde{\mathbb{P}}}
\newcommand{\tW}{\tilde{W}}
\newcommand{\tB}{\tilde{B}}
\newcommand{\1}{\mathbf{1}}
\renewcommand{\O}{\mathcal{O}}
\newcommand{\dt}{\Delta t}
\newcommand{\tr}{{\rm tr}}

\newcommand{\Xv}{X^{(v)}}
\newcommand{\Xvs}{X^{(v^*)}}
\newcommand{\Jv}{J^{(v)}}

\newcommand{\cG}{\mathcal{G}}
\newcommand{\cF}{\mathcal{F}}
\newcommand{\cLv}{\mathcal{L}^{(v)}}

\def\theequation{\thesection.\arabic{equation}}
\def\thetheorem{\thesection.\arabic{theorem}}

\renewcommand{\theequation}{\arabic{section}.\arabic{equation}}

\def\cprime{$'$}
\def\blue#1{\textcolor{blue}{#1}}

%
%
%

\begin{titlepage}

\begin{center}
\large \bf Bridge representation and modal-path approximation
\end{center}
\vspace{1cm}

\begin{center}
  Jiro Akahori\footnote{Supported by JSPS KAKENHI Grant Number $23330109$,
$24340022$, $23654056$ and $25285102$, and the project RARE -318984 (an FP7 Marie Curie IRSES).} \\
  Department of Mathematical Sciences, Ritsumeikan University \\
  Noji-higashi 1-1-1, Kusatsu, Shiga 525-8577, Japan \\
  e-mail: \textsf{akahori@se.ritsumei.ac.jp}
\end{center}

\vskip5mm

\begin{center}
  Xiaoming Song \\
  Department of Mathematics, Drexel University \\
  e-mail: \textsf{song@math.drexel.edu}
\end{center}

\vskip5mm

\begin{center}
  Tai-Ho Wang \\
  Department of Mathematics, Baruch College, CUNY \\
  1 Bernard Baruch Way, New York, NY10010 \\
  e-mail: \textsf{tai-ho.wang@baruch.cuny.edu}
\end{center}

\vskip 1cm
\begin{center}
\end{center}
\begin{center}
 {\bf Abstract}
 \end{center}
\vskip 0.2cm
The article shows a bridge representation for the joint density of a system of stochastic processes consisting of a Brownian motion with drift coupled with a correlated fractional Brownian motion with drift. As a result, a small time approximation of the joint density is readily obtained by substituting the conditional expectation under the bridge measure by a single path: the modal-path from the initial point to the terminal point. 

\vskip 5mm


\vskip 1cm


\end{titlepage}

%
%



\title[Bridge representation and modal-path approximation]{Bridge representation and modal-path approximation}

\begin{abstract}
The article shows a bridge representation for the joint density of a system of stochastic processes consisting of a Brownian motion with drift coupled with a correlated fractional Brownian motion with drift. As a result, a small time approximation of the joint density is readily obtained by substituting the conditional expectation under the bridge measure by a single path: the modal-path from the initial point to the terminal point. 
\end{abstract}

\author[J. Akahori]{Jiro Akahori}
\author[X. Song]{Xiaoming Song}
\author[T.-H Wang]{Tai-Ho Wang}

\address{Jiro Akahori \newline
Department of Mathematical Sciences \newline
Ritsumeikan University \newline
Noji-higashi 1-1-1, Kusatsu, Shiga 525-8577, Japan
}
\email{akahori@se.ritsumei.ac.jp}
\thanks{The first author is supported by JSPS KAKENHI Grant Number $23330109$,
$24340022$, $23654056$ and $25285102$, and the project RARE -318984 (an FP7 Marie Curie IRSES).}

\address{Xiaoming Song \newline
Department of Mathematics \newline
Drexel University
}
\email{song@math.drexel.edu}

\address{Tai-Ho Wang \newline
Department of Mathematics \newline
Baruch College, The City University of New York \newline
1 Bernard Baruch Way, New York, NY10010
}
\email{tai-ho.wang@baruch.cuny.edu}

\maketitle

%
%

\newcommand{\F}{\mathcal{F}}
\renewcommand{\th}{\tilde{h}}

%
%

\section{Introduction}
Stochastic modeling with long range dependence processes has nowadays become ubiquitous. Applications of such processes range from models for traffic, telecommunication, geophysics to finance. In this regard, among other continuous time processes, fractional Brownian motion is probably the most frequently used  base model for long range dependence due to its Gaussianity and close relationship with the classical Brownian motion.

In the field of quantitative finance, stochastic differential equations driven by fractional Brownian motions with different Hurst exponents are considered in option pricing theory in order to capture certain stylized facts observed in the market. For the model to be free of arbitrage opportunity, the underlying asset itself has to be driven by a Brownian motion, see for instance the discussions in Cheridito \cite{cheridito} and Rogers \cite{rogers-arb}. On the other hand, there are empirical evidences showing that the volatility of logarithmic returns of the underlying exhibits long range dependence, see for example Bollerslev and Mikkelsen \cite{bol-mik} and Granger and Hyng \cite{granger-hyng} for S\&P500 index and Tschernig \cite{tschernig} for foreign exchange rate. Thus, the price dynamic of the underlying is naturally modeled by a stochastic system driven by a mixture of Brownian and fractional Brownian motions. However, as probability density of such a model is concerned, to our knowledge, little is known in determining tractable analytic expressions or asymptotic expansions for the joint density; partly due to the lack of analytic tools from PDE theory.

In this paper, we consider the stochastic system consisting of a Brownian motion with drift coupled with a correlated fractional Brownian motion 
with Hurst parameter $ H \in (0,1) $ with drift. Modulo a Gaussian prefactor, we aim to derive a bridge representation for the joint density, see Theorem \ref{thm:bridge-rep}, which accordingly yields a small time asymptotic of the heat kernel type to the lowest order as shown in Theorem \ref{thm:small-time-exp}. The technique applied in the derivation of the bridge representation and hence the resulting small time asymptotics in our opinion is natural and straightforward since it is in a sense a direct generalization of the procedure in deriving similar representation in one dimensional case.

To obtain the bridge representation for the joint density, we follow the line of thought as in Rogers \cite{rogers} and Wang and Gatheral \cite{hk-prob} which we briefly summarize in the following. A general nondegenerate diffusion is transformed into a Brownian motion with drift by applying the Lamperti transformation. Girsanov's theorem is then applied to define a new equivalent measure so that the resulting process is driftless in the new measure.
Finally, modulo a Gaussian density, the bridge representation for the transition density is obtained by conditioning on the terminal point of Brownian motion, see for example Theorem 2 in \cite{hk-prob}. With this bridge representation, a small time asymptotic expansion of the transition density is readily obtained by expanding the Brownian bridge expectation around a deterministic path, the most-likely-path. See \cite{hk-prob} for more details. We remark that the trick of applying Lamperti transformation to unitize the diffusion coefficient in one dimensional case is generally not applicable in higher dimensions due to geometric obstructions.

Aside from some technical conditions, the technique of applying Girsanov's theorem to de-drift the coupled Brownian and fractional Brownian motions in the new measure is still applicable in our case. However, the integrands required in defining the Radon-Nikodym derivative for the new measure are more involved due to the appearance of the defining kernel of fractional Brownian motion, see \eqref{h-2-1} and \eqref{h-1-2}.

Modal-path approximation of the joint density is thus obtained by evaluating the bridge representation along a single deterministic path: the modal-path connecting the initial point and the terminal point. The rationale is as follows. Since in the new measure the two processes under consideration are respectively standard Brownian and fractional Brownian motions, in small time the densities of the corresponding bridges are peaked around their modes; hence the name modal-path. Moreover, at each point in time, the two processes are jointly Gaussian, therefore the modes are simply given by the expectations. For Brownian bridge, the modal-path is the straight line connecting the initial and terminal points of the bridge. As for fractional Brownian bridge, we use the form of Volterra bridge as in Baudoin and Coutin \cite{baudoin-coutin} to determine the modal-path, which in general is not a straight line. We remark that, as the Hurst exponent $H$ approaches one half, the modal-path gets closer to the straight line connecting the initial and terminal points. However, as $H$ approaches zero, the modal-path travels very quickly to the midpoint, stays around the midpoint till almost to the end, then travels very quickly to the terminal point, which in a sense creates a jump-like behaviour. See Remark \ref{rmk:volterra-bridge} for more details.

It is worth mentioning that recent papers by Baudoin and Ouyang \cite{baudoin-ouyang}, 
Inahama \cite{In1}, \cite{In2} and Yamada \cite{Yam}, studying a heat kernel type expansion for the joint density of solution to SDEs driven by fractional Brownian motions in small time. The driving fractional Brownian motions are assumed all of the same Hurst exponent $ H > 1/2 $ except for Inahama \cite{In1}, where the case $ H \in (1/3,1/2] $
is studied. Thus, it is conceivable that in logarithmic scale the lowest order in the expansion of the probability density is of $t^{2H}$ as $t\to0^+$. 

On the other hand, as closed form expression is concerned, Zeng, Chen, and Yang \cite{zeng-chen-yang} derived the density of a one dimensional Ornstein-Uhlenbeck process driven by fractional Brownian motion in closed form by solving a Fokker-Planck type of equation satisfied by the density function. The density in this case is unsurprisingly Gaussian, see (3.6) in Zeng, Chen, and Yang \cite{zeng-chen-yang}.

The rest of the paper is organized as follows. The main result of bridge representation is proved in Section \ref{sec:main-1}. Section \ref{sec:modal-path-approx} gives the modal-path approximation of the joint density and an error analysis of the approximation. Finally, the paper concludes with specific examples of the modal-path approximations. For reader's convenience, we review basics on fractional Brownian motion, fractional differentiation, and fractional integration in Section \ref{sec:prelim}.

%
%

\section{Model specification and change of probability measures} \label{sec:main-1}

Throughout the text, $B=\{B_t,t\in[0,\infty) \}$ and $W=\{W_t,t\in[0,\infty)\}$ denote independent standard Brownian motions defined on the complete filtered probability space $(\Omega,\cF, \P, \{\cF_t\}_{t\in[0,\infty)})$ satisfying the usual conditions. $B^H=\{B_t^H, t\in[0,\infty)\}$ is a fractional Brownian motion with Hurst exponent $H$ generated by $B$. In this paper we understand $B^H $ as the Volterra-Gaussian 
process given by
\begin{equation*} 
B_t^H=\int_0^t K_H(t,s)\,dB_s,
\end{equation*}
where $ K_H $ is given by
\eqref{K-H} or \eqref{K-H-1} in the Appendix. 

Let $ T >0 $. 
We shall make use of the following notations. Let $C([0,T])$ denote the space of continuous functions defined on $[0,T]$, and $C^\lambda([0,T])$ denote the space of H\"{o}lder continuous functions on $[0,T]$ of order $\lambda\in(0,1)$. The supremum norm and $C^\lambda$ norm are defined respectively as
\[
\Vert f\Vert_{T, \infty }=\sup_{0\leq t\leq T}|f(t)| \quad \mbox{if} \ f\in C([0,T]),
\]%
and%
\[
\left\| f\right\| _{T,\lambda }=\Vert f\Vert_{T,\infty }+\sup_{0\leq s <t\leq T}\frac{%
|f(t)-f(s )|}{|t-s |^{\lambda }}, \quad \mbox{if}\ f\in C^\lambda([0,T]).
\]%

The model and its assumptions are specified in Section \ref{sec:model}, followed by the proof of existence and uniqueness of solution and the regularity of sample paths. A change of probability measures is introduced and its validity is proved in Section \ref{sec:change-meas}. The bridge representation is shown in Section \ref{sec:bridge-rep}.

\bigskip
\subsection{The model}$\mbox{ } $ \label{sec:model}

Consider the two dimensional stochastic system
\begin{equation}\label{model}
\begin{cases}
   X_t =x_0+ \rho B_t + \sqrt{1-\rho^2} W_t +\int_0^t h_1(s,X_s,Y_s) ds, \\\\
    Y_t =y_0+ B^H_t + \int_0^th_2(s, X_s,Y_s) ds,
   \end{cases}
\end{equation}
where $(X_0,Y_0)=(x_0,y_0)$ is the initial point, $\rho \in (-1,1)$, and the two functions $h_1, h_2:[0,T]\times\mathbb{R}^2\to\mathbb{R}$ are deterministic. Hence, by construction, $X_t$ is a Brownian motion with drift $h_1$ and $Y_t$ is a fractional Brownian motion of Hurst exponent $H$ with drift $h_2$.

The following assumptions on $h_1$ and $h_2$ guarantee the existence and uniqueness of solution to \eqref{model}.
\begin{assumption}\label{a1} $\mbox{ }$ \hfill
\begin{itemize}
\item[(a)]
 The functions $h_1$ and $h_2$ are Lipschitz in $x,y$ uniformly for $t$. That is, there exists a constant $L>0$ such that
\begin{equation}\label{Lip}
|h_i(t, x_1,y_1)-h_i(t, x_2,y_2)|\leq L(|x_1-x_2|+|y_1-y_2|), \quad i = 1, 2,
\end{equation}
for all $t\in[0,T]$ and $(x_1,y_1),\, (x_2,y_2)\in\mathbb{R}^2$.
\item[(b)] \begin{itemize}
\item[(i)] If $H>\frac12$, there exist two constants $L>0$ and $\gamma\in(H-\frac{1}{2},\frac12)$ such that the function $h_1$ satisfies
\begin{equation*}
|h_1(t,0,0)|\leq L,\ \forall t\in[0,T],
\end{equation*} 
and the function $h_2$ satisfies
\begin{equation}\label{Hold}
|h_2(t, x,y)-h_2(s, x,y)|\leq L|t-s|^\gamma, \ \forall  s, t\in[0,T], \ \forall (x,y)\in\mathbb{R}^2,
\end{equation}
i.e., $h_2$ is H\"{o}lder continuous in $t$ of order $\gamma$ uniformly for $x$ and $y$.

\item[(ii)] If $H\leq \frac12$, there exists a constant $L>0$ such that
\begin{equation*}
|h_i(t,0,0)|\leq L,\ \forall s, t\in[0,T], \ i=1,2.
\end{equation*}
\end{itemize}
\end{itemize}
\end{assumption}

\begin{remark} The conditions  in Assumption \ref{a1} imply that the functions $h_1$ and $h_2$ satisfy the following 
linear growth condition:  there exists a constant $K>0$ such that
\begin{equation}\label{linear}
\vert h_i(t, x,y)\vert \leq K \left( 1 + \vert x \vert + \vert y \vert \right), \quad i = 1, 2,
\end{equation}
for all $t\in[0,T]$ and $(x,y)\in\mathbb{R}^2$.
\end{remark}
Since we consider small time asymptotic in this work, in the sequel, we always assume that $T<1$. We use the conventions of\[
\Vert\cdot\Vert_\infty=\Vert \cdot\Vert_{1,\infty}\quad \mbox{and}\quad  \Vert\cdot\Vert_\lambda=\Vert \cdot\Vert_{1,\lambda}.
\]

The following theorem establishes the existence and uniqueness of the solution to \eqref{model} and the regularity of solution trajectories under Assumption \ref{a1}.
\begin{theorem}\label{holder}
Let the conditions in Assumption \ref{a1} be satisfied. Then, there exists a positive constant $\delta$ such that the system \eqref{model} has a unique solution $(X, Y)$ when $T<\delta$. Moreover, the trajectories of $X$ and $Y$ satisfy $X\in C^{\frac{1}{2}-\epsilon}([0,T])$ and $Y\in C^{H-\epsilon}([0,T])$ almost surely for every $0<\epsilon  <\min\{\frac12, H\}$.
\end{theorem}
\begin{proof}
We use the contraction mapping theorem to prove the existence and uniqueness of the solution. Let $(x^i,y^i), i=1,2$, be two stochastic processes taking values in $C([0,T])$. Define
\begin{equation*}
\begin{cases}
   X_t^i =x_0+ \rho B_t + \sqrt{1-\rho^2} W_t +\int_0^t h_1(s, x^i_s,y^i_s) ds, \\\\
    Y_t^i =y_0+ B^H_t + \int_0^th_2(s,x^i_s,y^i_s) ds,
   \end{cases}
\end{equation*}
for each $i=1,2$. Assumption \ref{a1} implies that
\[
\Vert X^1-X^2\Vert_{T,\infty}+\Vert Y^1-Y^2\Vert_{T,\infty}\leq 2LT\left(\Vert x^1-x^2\Vert_{T,\infty}+\Vert y^1 - y^2 \|_{T,\infty} \right).
\]
We choose $\delta=\frac{1}{2L}$, and hence, by the contraction mapping theorem we obtain the existence and uniqueness of the solution in $C([0,T])$ for any $T<\delta$.

From \eqref{model} and \eqref{linear}, we get
\begin{equation*}
\vert X_t\vert\leq \vert x_0\vert+\rho \Vert B\Vert_\infty+\sqrt{1-\rho^2}\Vert W\Vert_\infty+K\int_0^t(1+\vert X_s\vert+\vert Y_s\vert)ds,
\end{equation*}
and
\begin{equation*}
\vert Y_t\vert\leq \vert y_0\vert+ \Vert B^H\Vert_\infty+K\int_0^t(1+\vert X_s\vert+\vert Y_s\vert)ds.
\end{equation*}
Thus, Gronwall's inequality implies that
\begin{eqnarray}
&& \Vert X\Vert_{T,\infty}+\Vert Y\Vert_{T,\infty} \label{e-2-4} \\
&\leq& \left(\vert x_0\vert+\vert y_0\vert+\rho \Vert B\Vert_\infty+\sqrt{1-\rho^2}\Vert W\Vert_\infty+\Vert B^H\Vert_\infty+2KT\right)e^{2KT}.  \nonumber
\end{eqnarray}
From \eqref{model}, \eqref{linear} and \eqref{e-2-4}, we observe that
\begin{eqnarray*}
&& \vert X_t-X_s\vert\leq \rho\vert B_t-B_s\vert+\sqrt{1-\rho^2}\vert W_t-W_s\vert+\left\vert \int_s^th_1(r, X_r,Y_r)dr\right\vert  \\
&&\leq \rho\vert B_t-B_s\vert+\sqrt{1-\rho^2}\vert W_t-W_s\vert+C\left(1+\Vert B\Vert_\infty+\Vert W\Vert_\infty+\Vert B^H\Vert_\infty\right)|t-s|,
\end{eqnarray*}
where $C$ here and in the sequel denotes a generic constant depending on $\vert x_0\vert$, $\vert y_0\vert$, $\rho$, $L$ and $K$ (the generic constant $C$ will depend on $H$ as well in the proofs of some later lemmas), and $C$ may vary from line to line. 

Then,  for any $\epsilon\in(0,\frac{1}{2})$, the following estimate can be obtained
\begin{equation}\label{Jan-2-5}
\Vert X\Vert_{T,\frac{1}{2}-\epsilon}\leq \rho\Vert B\Vert_{\frac{1}{2}-\epsilon}+\sqrt{1-\rho^2}\Vert W\Vert_{\frac{1}{2}-\epsilon}+C\left(1+\Vert B\Vert_\infty+\Vert W\Vert_\infty+\Vert B^H\Vert_\infty\right).
\end{equation}

Similarly, for any $\epsilon\in(0,H)$, the following estimate holds
\begin{equation}\label{Jan-2-6}
\Vert Y\Vert_{T, H-\epsilon}\leq \Vert B^H\Vert_{H-\epsilon}+C\left(1+\Vert B\Vert_\infty+\Vert W\Vert_\infty+\Vert B^H\Vert_\infty\right).
\end{equation}
The proof is completed.
\end{proof}

\bigskip
\subsection{Change of Measures}$\mbox{ } $ \label{sec:change-meas}

Next, we discuss a change of measures, where under the new measure $ X, Y $ become standard and 
fractional Brownian motions, respectively. 
Heuristically, the new measure $ \tP $
would be defined by
\begin{equation}\label{p-mea0}
   \frac{d\tP}{d\P}  = \exp\left\{-\int_0^{T} \th_1(t) dW_t - \frac12\int_0^{T} \th_1^2(t) dt - \int_0^{T} \th_2(t) dB_t - \frac12 \int_0^{T} \th_2^2(t) dt\right\}, \\
\end{equation}
where $ \th_1 $ and $ \th_2 $ are such
that 
\begin{equation}\label{h-2-1}
\int_0^tK_H(t,s)\tilde{h}_2(s)ds=\int_0^th_2(s, X_s,Y_s)ds,
\end{equation}
and
\begin{equation}\label{h-1-2}
\rho \,\tilde{h}_2(t)+\sqrt{1-\rho^2}\,\tilde{h}_1(t)=h_1(t, X_t,Y_t).
\end{equation}
The well-definedness of $ \th_2 $ in
\eqref{h-2-1} is established in 
Lemma \ref{inverse} in Appendix. 

The following lemma asserts that the two processes $\tilde h_1$ and $\tilde h_2$ satisfy the Novikov's condition.
\begin{lemma} \label{lma:novikov}
There exists a small $t_0\leq T$ such that the adapted processes $\tilde{h}_1$ and $\tilde{h}_2$ satisfy the Novikov's condition in $[0,t_0]$. That is,
\begin{equation}\label{novikov}
\mathbb{E}\left[\exp\left\{\frac{1}{2}\int_0^{t_0}\vert\tilde{h}_1(t)\vert^2dt+\frac{1}{2}\int_0^{t_0}\vert\tilde{h}_2(t)\vert^2dt\right\}\right]
<\infty.
\end{equation}
\end{lemma}
\begin{proof} $\mbox{ }$ \\
\underline{Case of $H\leq\frac{1}{2}$}:  From \eqref{equ-1-5}, \eqref{e-2-4} and the linear growth property \eqref{linear} of $h_2$, we get
\begin{eqnarray}\label{bound-h-2}
\vert \tilde{h}_2(t)\vert&=&c_H^{-1} t^{H-\frac{1}{2}}\left\vert\int_0^t(t-s)^{-\frac{1}{2}-H}s^{\frac{1}{2}-H}h_2(s, X_s,Y_s)ds\right\vert\nonumber\\
&\leq & C\left(1+\Vert B\Vert_\infty+\Vert W\Vert_\infty+\Vert B^H\Vert_\infty\right)t^{H-\frac{1}{2}}\int_0^t(t-s)^{-\frac{1}{2}-H}s^{\frac{1}{2}-H}ds\nonumber\\
&=&CB\left(\frac{1}{2}-H,\frac{3}{2}-H\right)\left(1+\Vert B\Vert_\infty+\Vert W\Vert_\infty+\Vert B^H\Vert_\infty\right)t^{\frac{1}{2}-H}\nonumber\\
&\leq& C \left(1+\Vert B\Vert_\infty+\Vert W\Vert_\infty+\Vert B^H\Vert_\infty\right),
\end{eqnarray}
where $B(\cdot,\cdot)$ is the Beta function.

From \eqref{h-1-2}, \eqref{bound-h-2}, the linear growth condition \eqref{linear} on $h_1$ and \eqref{e-2-4}, we obtain
\begin{equation*}
\vert \tilde{h}_1(t)\vert\leq C \left(1+\Vert B\Vert_\infty+\Vert W\Vert_\infty+\Vert B^H\Vert_\infty\right).
\end{equation*}
Thus, we obtain
\begin{equation}\label{Jan-2-13}
\int_0^t\vert(\tilde{h}_1(s)\vert^2+\vert\tilde{h}_2(s)\vert^2)ds\leq C\left(1+\Vert B\Vert_\infty^2+\Vert W\Vert_\infty^2+\Vert B^H\Vert_\infty^2\right)t.
\end{equation}
Therefore, the estimate \eqref{Jan-2-13} and the Fernique's theorem (see \cite{Fe}) imply \eqref{novikov} for a small enough $t_0$.

\underline{Case of $H>\frac{1}{2}$}: From \eqref{equ-1-3}, we can see that
\begin{equation}\label{Jan-2-14}
\tilde{h}_2(t)=\frac{c_H^{-1}}{\Gamma\left(\frac{3}{2}-H\right)}\left(a(t)+b(t)\right),
\end{equation}
where
\[
a(t)=t^{\frac{1}{2}-H}h_2(t, X_t,Y_t)+\left(H-\frac{1}{2}\right)\int_0^t\frac{h_2(t, X_t,Y_t)-h_2(s, X_s,Y_s)}{(t-s)^{H+\frac{1}{2}}}ds,
\]
and
\[
b(t)=\left(H-\frac{1}{2}\right)t^{H-\frac{1}{2}}\int_0^t\frac{(t^{\frac{1}{2}-H}-s^{\frac{1}{2}-H})h_2(s,X_s,Y_s)}{(t-s)^{H+\frac{1}{2}}}ds.
\]

For any small positive number $\epsilon<1-H<\frac{1}{2}$, from the conditions \eqref{Lip}, \eqref{Hold} and \eqref{linear} on $h_2$, \eqref{e-2-4}, \eqref{Jan-2-5} and \eqref{Jan-2-6} it follows that
\begin{align}\label{Jan-2-15}
|a(t)|&\leq C\left(1+\Vert B\Vert_\infty+\Vert W\Vert_\infty+\Vert B^H\Vert_\infty\right)t^{\frac{1}{2}-H}+C\int_0^t(t-s)^{\gamma-H-\frac12}ds\nonumber\\
&\quad +C\left(1+\Vert B\Vert_\infty+\Vert W\Vert_\infty+\Vert B^H\Vert_\infty+\Vert B\Vert_{\frac{1}{2}-\epsilon}+\Vert W\Vert_{\frac{1}{2}-\epsilon}+\Vert B^H\Vert_{H-\epsilon}\right)\nonumber\\
&\qquad\times\int_0^t(t-s)^{-H-\epsilon}ds\nonumber\\
&\leq C\left(1+\Vert B\Vert_\infty+\Vert W\Vert_\infty+\Vert B^H\Vert_\infty\right)t^{\frac{1}{2}-H}+Ct^{\gamma-H+\frac12}\nonumber\\
&\quad +C\left(1+\Vert B\Vert_\infty+\Vert W\Vert_\infty+\Vert B^H\Vert_\infty+\Vert B\Vert_{\frac{1}{2}-\epsilon}+\Vert W\Vert_{\frac{1}{2}-\epsilon}+\Vert B^H\Vert_{H-\epsilon}\right)t^{1-H-\epsilon}\nonumber\\
&\leq C\left(1+\Vert B\Vert_\infty+\Vert W\Vert_\infty+\Vert B^H\Vert_\infty\right)t^{\frac{1}{2}-H}\nonumber\\
&\quad +C\left(1+\Vert B\Vert_\infty+\Vert W\Vert_\infty+\Vert B^H\Vert_\infty+\Vert B\Vert_{\frac{1}{2}-\epsilon}+\Vert W\Vert_{\frac{1}{2}-\epsilon}+\Vert B^H\Vert_{H-\epsilon}\right).
\end{align}
For the term $b(t)$, we shall apply the fact that the integral $\int_0^1\frac{u^{\frac{1}{2}-H}-1}{(1-u)^{H+\frac{1}{2}}}du=\alpha_H$ is a finite number depending only on $H$. Using the linear growth condition \eqref{linear} on $h_2$, \eqref{e-2-4} and the change of variable $u = \frac st$, we obtain
\begin{align}\label{Jan-2-16}
|b(t)|&\leq C\left(1+\Vert B\Vert_\infty+\Vert W\Vert_\infty+\Vert B^H\Vert_\infty\right)t^{\frac{1}{2}-H}\int_0^1\frac{u^{\frac{1}{2}-H}-1}{(1-u)^{H+\frac{1}{2}}}du\nonumber\\
&= C\alpha_H\left(1+\Vert B\Vert_\infty+\Vert W\Vert_\infty+\Vert B^H\Vert_\infty\right)t^{\frac{1}{2}-H}.
\end{align}
From \eqref{Jan-2-14} - \eqref{Jan-2-16}, we can show that
\begin{align}\label{Jan-2-17}
|\tilde{h}_2(t)|&\leq C\left(1+\Vert B\Vert_\infty+\Vert W\Vert_\infty+\Vert B^H\Vert_\infty\right)t^{\frac{1}{2}-H}\nonumber\\
&\quad +C\left(1+\Vert B\Vert_\infty+\Vert W\Vert_\infty+\Vert B^H\Vert_\infty+\Vert B\Vert_{\frac{1}{2}-\epsilon}+\Vert W\Vert_{\frac{1}{2}-\epsilon}+\Vert B^H\Vert_{H-\epsilon}\right).
\end{align}
From \eqref{h-1-2}, \eqref{Jan-2-17}, the linear growth condition \eqref{linear} on $h_1$ and \eqref{e-2-4}, we obtain
\begin{align}\label{bound-h-1}
\vert \tilde{h}_1(t)\vert&\leq C\left(1+\Vert B\Vert_\infty+\Vert W\Vert_\infty+\Vert B^H\Vert_\infty\right)t^{\frac{1}{2}-H}\nonumber\\
&\quad +C\left(1+\Vert B\Vert_\infty+\Vert W\Vert_\infty+\Vert B^H\Vert_\infty+\Vert B\Vert_{\frac{1}{2}-\epsilon}+\Vert W\Vert_{\frac{1}{2}-\epsilon}+\Vert B^H\Vert_{H-\epsilon}\right).
\end{align}
Then, it is easy to see that
\begin{eqnarray}\label{ss-2-19}
&&\int_0^t\vert(\tilde{h}_1(s)\vert^2+\vert\tilde{h}_2(s)\vert^2)ds\nonumber\\
&\leq& C\left(1+\Vert B\Vert_\infty^2+\Vert W\Vert_\infty^2+\Vert B^H\Vert_\infty^2\right)t^{2-2H}\nonumber\\
&&+C\left(1+\Vert B\Vert_\infty^2+\Vert W\Vert_\infty^2+\Vert B^H\Vert_\infty^2+\Vert B\Vert_{\frac{1}{2}-\epsilon}^2+\Vert W\Vert_{\frac{1}{2}-\epsilon}^2+\Vert B^H\Vert_{H-\epsilon}^2\right)t.\nonumber\\
\end{eqnarray}
Therefore, from the estimate \eqref{ss-2-19} and the Fernique's theorem (see \cite{Fe}), we conclude that there exists a small enough $t_0$ such that \eqref{novikov} holds.
\end{proof}

We let $T\leq t_0$ and restrict ourselves to the small interval $[0,T]$ hereafter for convenience. By Lemma \ref{lma:novikov}, we define the equivalent probability measure $\tP$ via the Radon-Nikodym derivative by \eqref{p-mea0}. 
Thus, by applying Girsanov's theorem, we obtain that in the new probability measure $\tP$, $W$ and $B$ become two Brownian motions with drifts determined by $\tilde h_1$ and $\tilde h_2$ respectively. We summarize the result in the following lemma.
\begin{lemma}\label{Girsanov}
Under the probability measure $\tP$, the processes $\tilde{W}=\{\tilde{W}_t=W_t+\int_0^t\tilde{h}_1(s)ds,\,t\in[0,T]\}$ and $\tilde{B}=\{\tilde{B}_t=B_t+\int_0^t\tilde{h}_2(s)ds,\,t\in[0,T]\}$ become two independent Brownian motions, and the process $\tilde{B}^H=\{\tilde{B}^H_t=B^H_t+\int_0^tK_H(t,s)\tilde{h}_2(s)ds,\,t\in[0,T]\}$ becomes a fractional Brownian motion. Henceforth, in the $\tP$-measure $W$ and $B$ become two Brownian motions with drifts respectively.
\end{lemma}

\subsection{Bridge representation of the joint density under $\P$} $\mbox{ } $ \label{sec:bridge-rep}

The purpose of this section is to show the bridge representation \eqref{eqn:bridge-rep} of the joint density of $(X_T, Y_T)$.

\begin{theorem}(Bridge representation of the joint density) \label{thm:bridge-rep} \\
Let $X$ and $Y$ respectively be Brownian and fractional Brownian motions with drift satisfying \eqref{model} and initial condition $(x_0,y_0)$. The joint density $p_T(x,y|x_0,y_0)$ of $(X_T, Y_T)$ at time $T$ has the bridge representation
\begin{eqnarray}\label{eqn:bridge-rep}
&& p_T(x,y|x_0,y_0)\nonumber\\ 
&=& \phi(x - x_0,y - y_0) \tilde{\mathbb{E}}_{x,y} \left[e^{\int_0^{T} \th_1(t) d\tilde{W}_t - \frac12\int_0^{T} \th_1^2(t) dt + \int_0^{T} \th_2(t) d\tilde{B}_t - \frac12 \int_0^{T} \th_2^2(t) dt}\right],
\end{eqnarray}
where $\tilde{\mathbb{E}}_{x,y}[\cdot]$ denotes the conditional expectation with respect to the probability measure $\tP$ under which $X$ and $Y$ are standard Brownian and fractional Brownian bridges respectively conditioned on the terminal point $(X_T, Y_T) = (x, y)$,
$\phi$ is the bivariate Gaussian density
\bea
&& \phi(\xi,\eta) = \frac1{2\pi T^{H + \frac12} \sqrt{1 - \rho^2 \kappa_H^2}} \times \label{eqn:gaussian-density} \\
&& \exp\left\{-\frac1{2(1-\rho^2\kappa_H^2)}\left[\left(\frac{\xi}{\sqrt T}\right)^2 - 2 \rho\kappa_H \left(\frac{\xi}{\sqrt T}\right)\left(\frac{\eta}{T^H}\right) + \left(\frac{\eta}{T^H}\right)^2 \right]\right\}. \nonumber
\eea
and the processes $\tilde h_1$ and $\tilde h_2$ are determined by \eqref{h-2-1} and \eqref{h-1-2}. The constant $\kappa_H$ is defined by 
\begin{equation} \label{eqn:kappaH}
\kappa_H = c_H \,\frac{B\left(\frac32-H,H+\frac12\right)}{H+\frac12},
\end{equation}
where $c_H$ is the constant appearing in \eqref{K-H}.
\end{theorem}

To start with, we determine the dynamic of $(X_t,Y_t)$ in the $\tP$-measure. Recall from \eqref{h-2-1} and \eqref{h-1-2} that we can rewrite the processes $X$ and $Y$ as
\begin{eqnarray}\label{til-X}
  X_t &=& x_0+ \rho B_t + \sqrt{1-\rho^2}W_t + \int_0^th_1(s, X_s,Y_s) ds \nonumber\\
  &=& x_0+ \rho \tB_t + \sqrt{1-\rho^2}\tW_t +\int_0^t \left[ h_1(s, X_s,Y_s) - \rho\th_2(s) - \sqrt{1-\rho^2}\,\th_1(s) \right] ds\nonumber\\
  &=& x_0+ \rho \tB_t + \sqrt{1-\rho^2}\tW_t
\end{eqnarray}and
\begin{eqnarray}\label{til-Y}
  Y_t &=& y_0 + B^H_t + \int_0^t h_2(s, X_s,Y_s) ds \nonumber\\
  &=& y_0 + \int_0^t K_H(t,s)dB_s + \int_0^t h_2(s, X_s,Y_s) ds\nonumber \\
  &=& y_0 + \int_0^t K_H(t,s)d\tB_s - \int_0^t K_H(t,s) \th_2(s) ds + \int_0^t h_2(s, X_s,Y_s) ds\nonumber \\
  &=& y_0 + \tB^H_t.
\end{eqnarray}
Thus, under the probability measure $\tP$, the process $X$ is a Brownian motion and $Y$ is a fractional Brownian motion. In particular, $(X_t, Y_t)$ is jointly Gaussian in $\tP$.
The following lemma is required in determining the covariance matrix of $(X_t, Y_t)$ in $\tP$.
\begin{lemma} \label{int}
The integral $\int_0^tK_H(t,u)du$ is given explicitly as
\begin{equation}\label{int-KH}
\int_0^tK_H(t,u)du
=\kappa_Ht^{H+\frac12},
\end{equation}
where the constant $\kappa_H$ is given in \eqref{eqn:kappaH}. 
\end{lemma}
\begin{proof} 
Recall the function $K_H(t,s)$ in \eqref{K-H-1}:
\beaa
  K_H(t,s) &=&c_H \left[ \left(\frac ts \right)^{H-\frac12} (t - s)^{H-\frac12}
  - (H-\frac12) s^{\frac12-H} \int_s^t u^{H-\frac32} (u-s)^{H-\frac12} du \right].
\eeaa

By changing of variables $s=tr$, one can calculate
\beaa
  && \int_0^t  \left(\frac ts \right)^{H-\frac12} (t - s)^{H-\frac12}ds =t^{H+\frac12}\int_0^1r^{\frac12-H}(1-r)^{H-\frac12}dr=  t^{H+\frac12}B\left(\frac32-H,H+\frac12\right).
  \eeaa
By changing the order of the integrals and changing of variables $s=ur$, we have
\begin{eqnarray*}
&& \int_0^t s^{\frac12-H} \int_s^t u^{H-\frac32} (u-s)^{H-\frac12} du \,ds = \int_0^t\int_0^us^{\frac12-H}u^{H-\frac32} (u-s)^{H-\frac12}ds\,du \\
&=&\int_0^t u^{H-\frac12}\int_0^1r^{\frac12-H}(1-r)^{H+\frac12}dr\,du
= \frac{t^{H+\frac12}B\left(\frac32-H,H+\frac12\right)}{H+\frac12}.
\end{eqnarray*}
Thus, we obtain \eqref{int-KH}.
\end{proof}
Now we are in position to complete the proof of bridge representation for the joint density $p_T$.
\begin{proof}({\bf Proof of Theorem \ref{thm:bridge-rep}}) \\
Let $\tilde{\mathbb{E}}$ denote the expectation with respect to the probability $\tP$. Since $X_t$ and $Y_t$ are Brownian and fractional Brownian motions in $\tP$ respectively, we have
\[
\tE{X_t}=x_0,\quad \tE{Y_t}=y_0, \; \forall t \in [0,T],
\]
and, for $s, t \in [0,T]$,
\[
\tE{(X_s-x_0)(X_t-x_0)}=s\wedge t,\quad \tE{(Y_s-y_0)(Y_t-y_0)} = R_H(s,t),
\]
where $R_H$ is the autocovariance function for fractional Brownian motion as given in \eqref{covariance}.
The covariance between $X_t$ and $Y_t$ is determined by applying It\^{o} isometry as
\begin{eqnarray}
\tE{(X_t-x_0)(Y_s-y_0)}&=& \tE{\left(\rho \tB_t + \sqrt{1 - \rho^2} \tW_t\right) \int_0^s K_H(s,u) d\tB_u}\nonumber \\
  &=& \rho \int_0^{s\wedge t} K_H(s,u) du,  \ \forall s, t\in[0,T].
\end{eqnarray}
Thus, the joint density $\tilde p_T$ of the bivariate Gaussian variable $(X_T, Y_T)$ in $\tP$ is given by
\beaa
&& \tilde p_T(x,y|x_0, y_0) = \frac1{2\pi\sqrt{|\mathbf{\Sigma}(T)|}} e^{-\frac12\bx'\mathbf{\Sigma}(T)^{-1}\bx}  \\
&=& \frac1{2\pi T^{H + \frac12} \sqrt{1 - \rho^2 \kappa_H^2}} \times \\
&& \exp\left\{-\frac1{2(1-\rho^2\kappa_H^2)}\left[\left(\frac{x - x_0}{\sqrt T}\right)^2 - 2 \rho\kappa_H \left(\frac{x - x_0}{\sqrt T}\right)\left(\frac{y - y_0}{T^H}\right) + \left(\frac{y - y_0}{T^H}\right)^2 \right]\right\}\\
&=&\phi(x-x_0,y-y_0),
\eeaa
where $\mathbf{\Sigma}(T)$ denotes the covariance matrix of $(X_T, Y_T)$ given by
\begin{equation*}
\mathbf{\Sigma}(T)=\left(\begin{array}{cc}T&\rho\int_0^TK_H(T,u)du\\ \rho\int_0^TK_H(T,u)du &R_H(T,T)=T^{2H}\end{array}\right)=\left(\begin{array}{cc}T&\rho\kappa_HT^{H+\frac12}\\ \rho\kappa_HT^{H+\frac12} &T^{2H}\end{array}\right).
\end{equation*}

Recall from \eqref{p-mea0} we have
\begin{eqnarray*}
 \frac{d\P}{d\tP}&=& \exp\left\{\int_0^{T} \th_1(t) dW_t + \frac12\int_0^{T} \th_1^2(t) dt + \int_0^{T} \th_2(t) dB_t + \frac12 \int_0^{T} \th_2^2(t) dt\right\} \\
 &=&\exp\left\{\int_0^{T} \th_1(t) d\tilde{W}_t - \frac12\int_0^{T} \th_1^2(t) dt + \int_0^{T} \th_2(t) d\tilde{B}_t - \frac12 \int_0^{T} \th_2^2(t) dt\right\}.
\end{eqnarray*}
Hence, for any bounded and continuous function $f$ defined on $\mathbb{R}^2$, we have
\beaa
  && \int p_T(x,y|x_0, y_0) f(x,y) dx dy \\
  &=& \E{f(X_T,Y_T)} = \tE{f(X_T,Y_T) \frac{d\P}{d\tP}} \\
  &=& \tE{f(X_T,Y_T) e^{\int_0^{T} \th_1(t) d\tilde{W}_t - \frac12\int_0^{T} \th_1^2(t) dt + \int_0^{T} \th_2(t) d\tilde{B}_t - \frac12 \int_0^{T} \th_2^2(t) dt}} \\
  &=& \int \tilde{p}_T(x,y|x_0, y_0) f(x,y) \tilde{\mathbb{E}}_{x,y} \left[e^{\int_0^{T} \th_1(t) d\tilde{W}_t - \frac12\int_0^{T} \th_1^2(t) dt + \int_0^{T} \th_2(t) d\tilde{B}_t - \frac12 \int_0^{T} \th_2^2(t) dt}\right] dx dy,
\eeaa
where $\tilde{\mathbb{E}}_{x,y}$ is the conditional expectation conditioned on the terminal point $(X_T, Y_T) = (x, y)$.

 Finally, since $f$ is arbitrary, we obtain the bridge representation
\beaa
 p_T(x,y|x_0,y_0) &=& \tilde{p}_T(x,y|x_0,y_0) \tilde{\mathbb{E}}_{x,y} \left[e^{\int_0^{T} \th_1(t) d\tilde{W}_t - \frac12\int_0^{T} \th_1^2(t) dt + \int_0^{T} \th_2(t) d\tilde{B}_t - \frac12 \int_0^{T} \th_2^2(t) dt}\right]\\
&=&\phi(x-x_0,y-y_0)\tilde{\mathbb{E}}_{x,y} \left[e^{\int_0^{T} \th_1(t) d\tilde{W}_t - \frac12\int_0^{T} \th_1^2(t) dt + \int_0^{T} \th_2(t) d\tilde{B}_t - \frac12 \int_0^{T} \th_2^2(t) dt}\right].
\eeaa
The proof is completed.
\end{proof}

%
%

\section{Modal-path approximation} \label{sec:modal-path-approx}
The bridge representation \eqref{eqn:bridge-rep} of the joint density given in Theorem \ref{thm:bridge-rep} albeit succinct is hard to calculate in practice; owing to the complexity in defining the processes $\tilde h_1$, $\tilde h_2$ and the involvement of the stochastic integrals with respect to Brownian motions in the new measure $\tP$. In this section, we approximate the following conditional expectation under bridge measure in Theorem \ref{thm:bridge-rep}
\begin{equation}\label{condi-expo}
 \tilde{\mathbb{E}}_{x,y} \left[e^{\int_0^{T} \th_1(t) d\tilde{W}_t - \frac12\int_0^{T} \th_1^2(t) dt + \int_0^{T} \th_2(t) d\tilde{B}_t - \frac12 \int_0^{T} \th_2^2(t) dt}\right]
\end{equation}
by evaluating the integrand along the modal-path, thus the term ``modal-path approximation", and provide an error estimate of the modal-path approximation. The idea is to replace the processes $X$ and $Y$ in the bridge representation by their expectations in the new measure $\tP$. Since $X$ and $Y$ are Brownian and fractional Brownian bridges respectively in the $\tP$-measure, the expectations consist of the mode of the joint density of $X_t$ and $Y_t$ that are easily obtained as in Remark \ref{rmk:volterra-bridge}. We summarize the result in Theorem \ref{thm:small-time-exp}.

\subsection{The law of $(X_t, Y_t)$ conditioned on its terminal point in the $\tP$-measure} $\mbox{}$
Let us characterize the law of $(X_t, Y_t)$ conditioned on its terminal point in the $\tP$-measure in this subsection.
\begin{lemma}
Define the processes $(X_t^{x,y}, Y_t^{x,y})$ by
\begin{equation} \label{eqn:XY-bridge}
\left[\begin{array}{c} X_t^{x,y}\\ Y_t^{x,y}\end{array}\right] =
\left[\begin{array}{c} X_t\\ Y_t\end{array} \right] + \mathbf{\Sigma}(t;T) \mathbf{\Sigma}(T)^{-1}\left[\begin{array}{c} x-X_T\\ y-Y_T\end{array}\right],
\end{equation}
where the matrices $\mathbf{\Sigma}(t;T)$ and $\mathbf{\Sigma}(T)$ are given by
\begin{equation}\label{A_t}
\mathbf{\Sigma}(T) =\left[\begin{array}{cc} Cov(X_T,X_T)&Cov(X_T,Y_T)\\Cov(X_T,Y_T)&Cov(Y_T,Y_T) \end{array}\right]= \left[\begin{array}{cc} T & \rho_H T^{H + \frac12} \\
\rho_H T^{H + \frac12} & T^{2H}
\end{array}\right]
\end{equation}
and
\begin{equation}\label{B_t}
\mathbf{\Sigma}(t;T) =\left[\begin{array}{cc} Cov(X_t,X_T)&Cov(X_t,Y_T)\\Cov(Y_t,X_T)&Cov(Y_t,Y_T) \end{array}\right]= \left[\begin{array}{cc}
t & \rho\int_0^tK_H(T,u)du \\ \rho_H t^{H + \frac12} & R_H(t,T)
\end{array}\right].
\end{equation}
Note that $R_H(t,T)$ is the autocovariance function of fractional Brownian motion defined in \eqref{covariance}.  
Then, the joint process $(X_t^{x,y}, Y_t^{x,y})$ under the probability measure $\tilde{\mathbb{P}}$ has the same distribution as $(X_t, Y_t)$ under the conditional probability measure $\tP_T^{x,y} =  \tP\left[\cdot|X_T = x, Y_T = y\right]$.
\end{lemma}
\begin{proof}
Straightforward application of Lemma \ref{conditional} in Appendix.
\end{proof}

For notational simplicity, we denote $x_t^{x,y}=\tE{X_s^{x,y}}$ and $y_t^{x,y}=\tE{Y_s^{x,y}}$, and we use the following notations hereafter:
\beaa
\bar\rho = \sqrt{1 - \rho^2}, \quad \rho_H = \rho \kappa_H, \quad \bar\rho_H = \sqrt{1 - \rho_H^2}.
\eeaa
 By applying \eqref{eqn:XY-bridge} and straightforward calculations, we obtain the explicitly expression for the modal-path as
\bea
&&x_t^{x,y}= \tEof{X_t^{x,y}} = x_0 + m_{11}(t;T)(x - x_0) + m_{12}(t;T)(y - y_0), \label{eqn:EXtxy} \\
&& y_t^{x,y}=\tEof{Y_t^{x,y}} = y_0 + m_{21}(t;T)(x - x_0) + m_{22}(t;T)(y - y_0), \label{eqn:EYtxy}
\eea
where
\bea
&& m_{11}(t; T) = \frac1{\bar\rho_H^2}\left(\frac tT - \frac{\rho\rho_H}{T^{H + \frac12}} \int_0^tK_H(T,s)ds \right),\label{m11} \\
&& m_{12}(t;T) = \frac1{\bar\rho_H^2}\left(-\rho_H \frac{t}{T^{H+\frac12}} + \frac{\rho}{T^{2H}} \int_0^tK_H(T,s)ds \right),\label{m12} \\
&& m_{21}(t;T) = \frac{\rho_H}{\bar\rho_H^2}\left( \frac{t^{H + \frac12}}T - \frac{R_H(t,T)}{T^{H+\frac12}} \right),\label{m21} \\
&& m_{22}(t;T) = \frac1{\bar\rho_H^2}\left(-\rho_H^2 \left\{\frac tT\right\}^{H + \frac12} + \frac{R_H(t,T)}{T^{2H}}  \right).\label{m22}
\eea
See Figures \ref{fig:modal-path1} and \ref{fig:modal-path2} for plots of modal-paths in various cases. We remark that, if $H = \frac12$, then $m_{12}(t;T) = m_{21}(t;T) = 0$ and $m_{11}(t;T) = m_{22}(t;T) = \frac tT$. Therefore, \eqref{eqn:EXtxy} and \eqref{eqn:EYtxy} reduce to
\beaa
&& \tEof{X_t^{x,y}} = x_0 + \frac tT(x - x_0), \\
&& \tEof{Y_t^{x,y}} = y_0 + \frac tT(y - y_0),
\eeaa
which is simply the straight line connecting $(x_0, y_0)$ and $(x,y)$ as expected even though $X_t^{x,y}$ and $Y_t^{x,y}$ are correlated. On the other hand, if $H \neq \frac12$ but $\rho = 0$, then $\rho_H = 0$ and $\bar\rho_H = 1$. It follows that
\beaa
&& \tEof{X_t^{x,y}} = x_0 + \frac tT(x - x_0), \\
&& \tEof{Y_t^{x,y}} = y_0 + \frac{R_H(t,T)}{T^{2H}}(y - y_0).
\eeaa
In either case, there are no interactions between $x_t^{x,y}$ and $y_t^{x,y}$. 
\begin{remark} \label{rmk:volterra-bridge}
We present plots of modal-paths with various Hurst exponents $H$ and correlation coefficients $\rho$ in Figures \ref{fig:modal-path1} and \ref{fig:modal-path2}. Terminal time is set as $T=1$, initial and terminal points are chosen as $(x_0, y_0)=(0,0)$ and $(x_T, y_T)=(1,1)$ respectively. As one can see in the plots, when the driving Brownian motions are positively correlated ($\rho \geq 0$), the smaller the Hurst exponent, the curvier the modal-path. When $H$ is close to zero, we observe a jump-like behaviour in the modal-path for all the $\rho$'s. When $H \approx \frac12$, the modal-paths all look like straight lines independent of the values of $\rho$. The negatively correlated case ($\rho < 0$) behaves much more differently than the positive cases when $H$ is away from one half. 
\begin{figure}[ht!]
\begin{center}
\includegraphics[width=7cm, height=8cm]{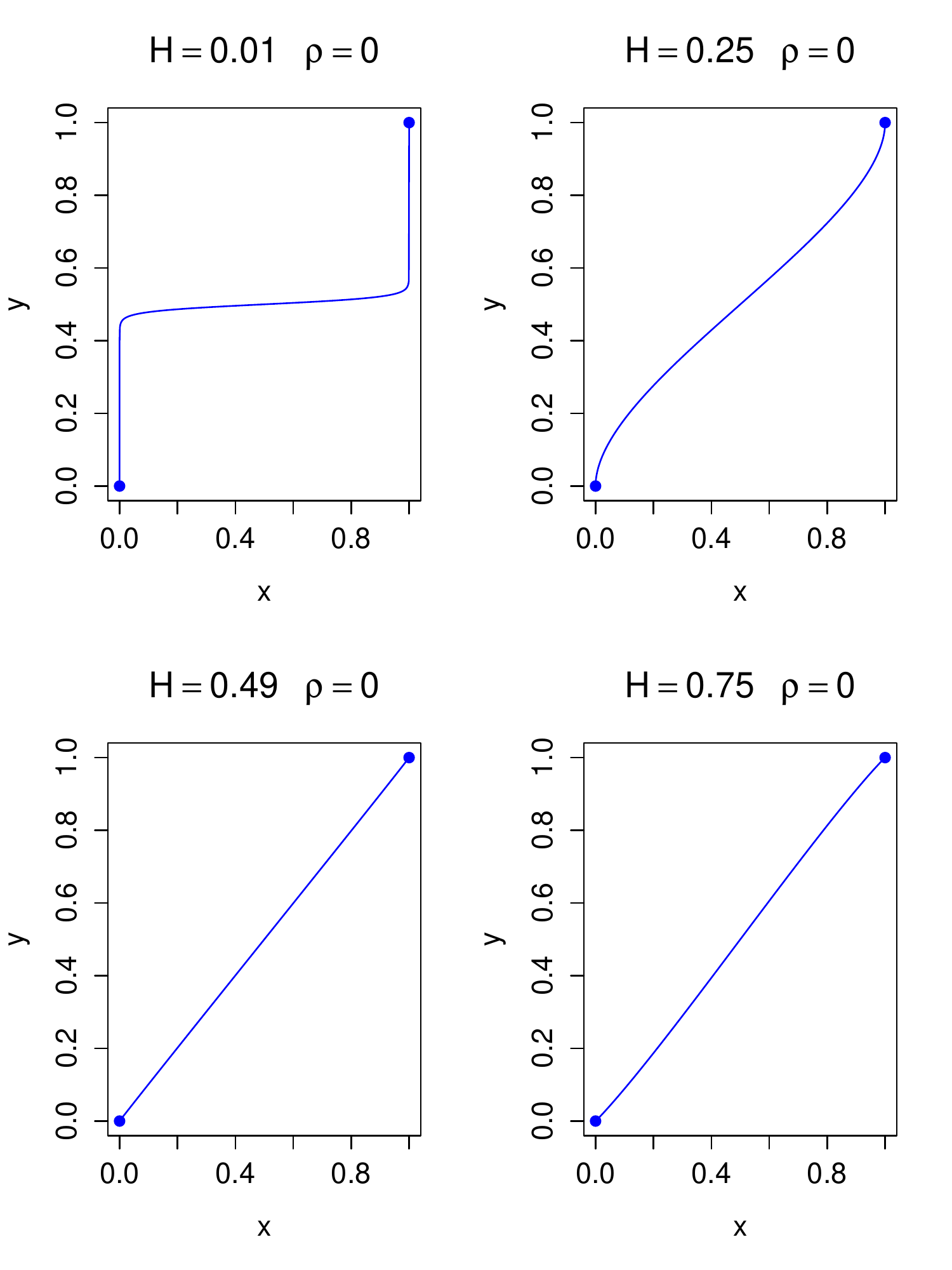} \;
\includegraphics[width=7cm, height=8cm]{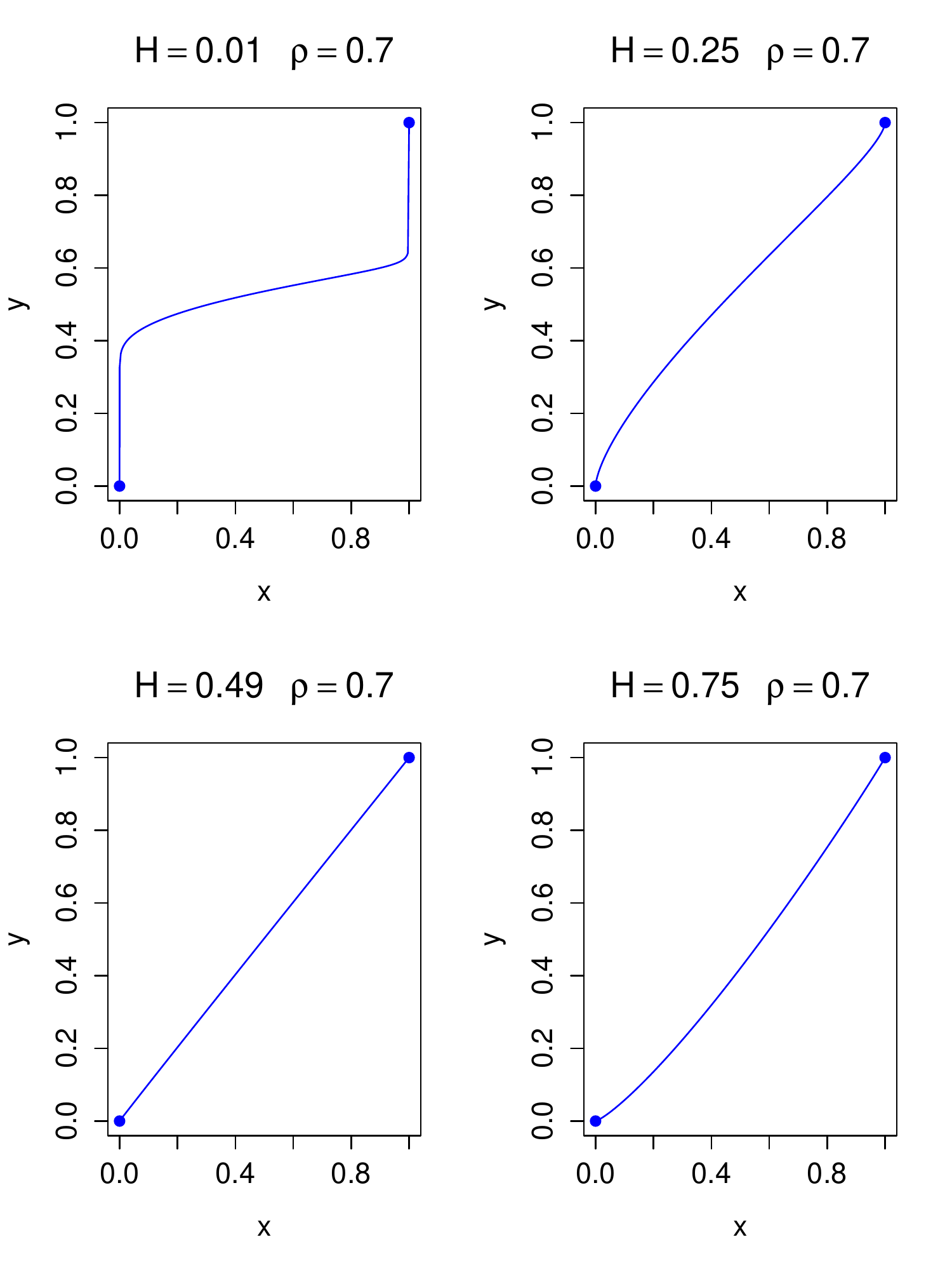}
\end{center}
\caption{The plots of modal-paths from $(x_0,y_0) = (0,0)$ to $(x,y)=(1,1)$ within the time interval $[0,1]$ with $\rho = 0$, $0.7$ and Hurst exponents $H = 0.01$, $0.25$, $0.49$, and $0.75$.}
\label{fig:modal-path1}
\end{figure}
\begin{figure}[ht!]
\begin{center}
\includegraphics[width=7cm, height=8cm]{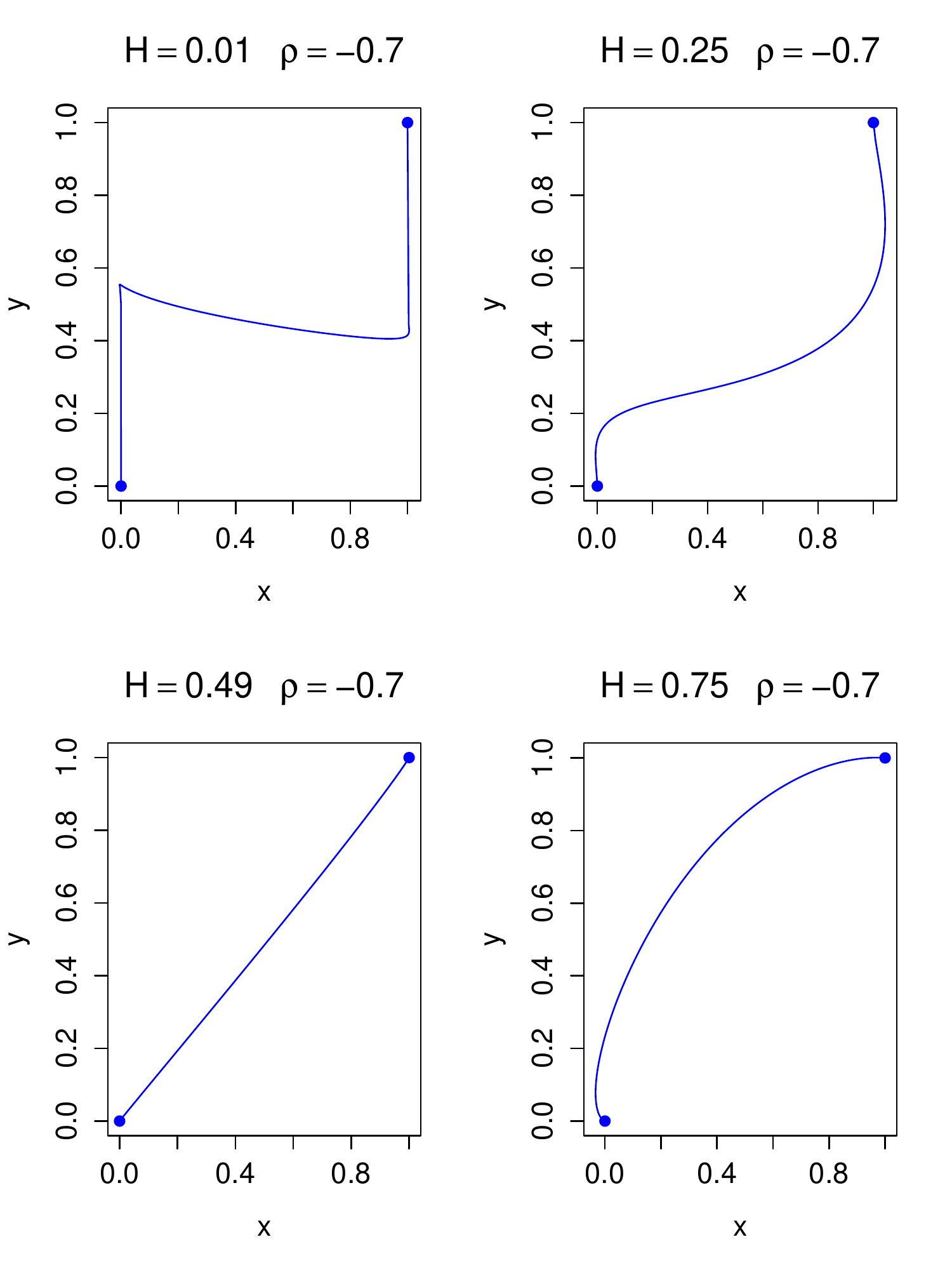} \;
\includegraphics[width=7cm, height=8cm]{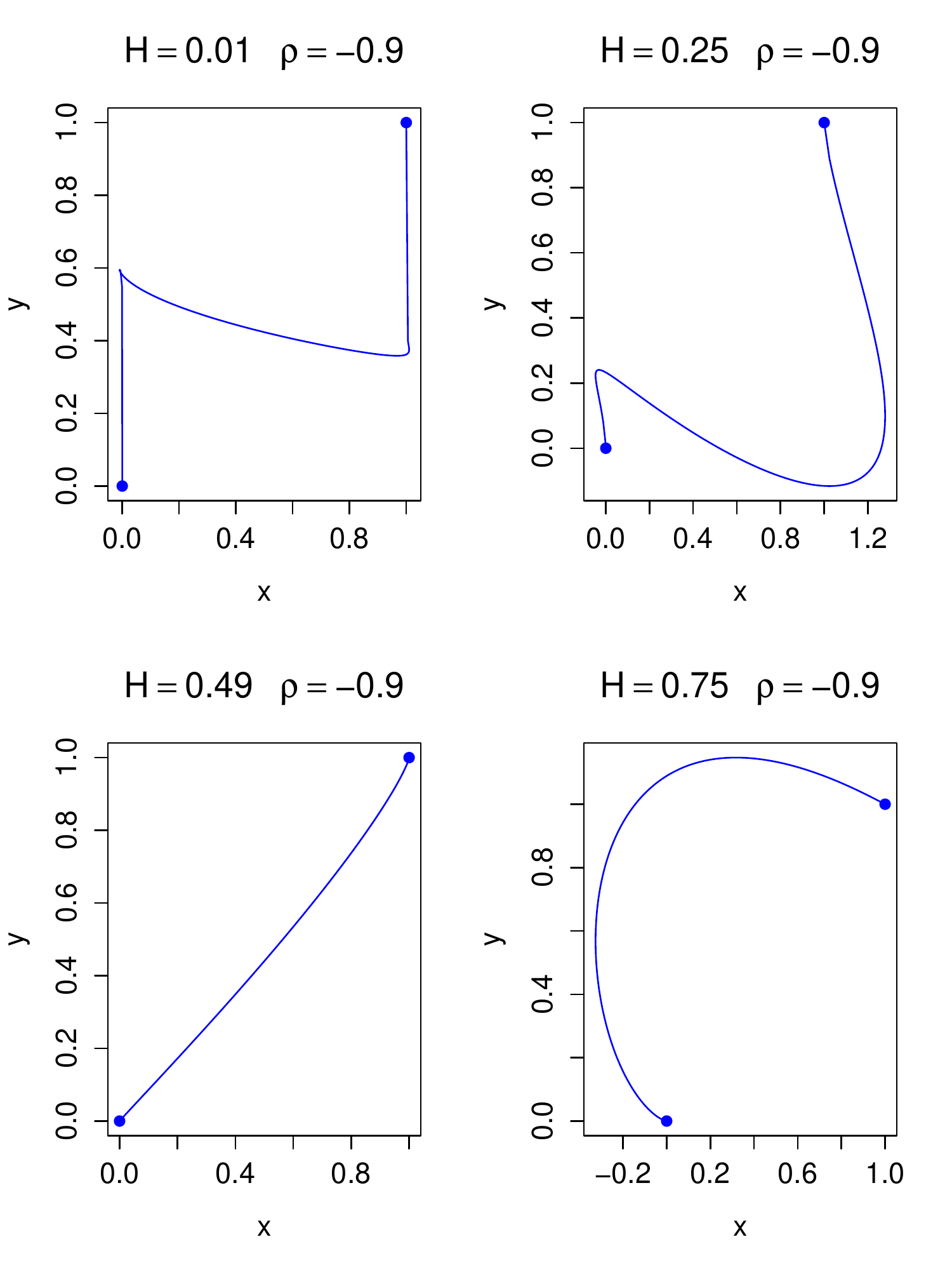}
\end{center}
\caption{The plots of modal-paths from $(x_0,y_0) = (0,0)$ to $(x,y)=(1,1)$ within the time interval $[0,1]$ with $\rho = -0.7$, $-0.9$ and Hurst exponents $H = 0.01$, $0.25$, $0.49$, and $0.75$.}
\label{fig:modal-path2}
\end{figure}
\end{remark}

\subsection{Small time asymptotic of the joint density} $\mbox{}$

Define for $i = 1, 2$
\beaa
&& \bar h_i(s) =h_i(s,x_s^{x,y},y_s^{x,y})
\eeaa
and the $\hat h_i$'s by
\bea
&& \rho\hat{h}_2(t) + \sqrt{1-\rho^2}\hat{h}_1(t) = \bar h_1(t), \label{hat-2} \\
&& \hat h_2(t) = \mathcal{K}_H^{-1}\left(\int_0^\cdot \bar h_2(s) ds\right)(t). \label{hat-1}
\eea
In other words, $\bar h_i(s)$ represents the value of $h_i$ evaluated along the modal path and the $\hat h_i$'s are solution to the system of equations similar to \eqref{h-2-1} and \eqref{h-1-2} except that the $X_s$ and $Y_s$ are substituted by the modal path. The well-definedness of $\hat h_1$ and $\hat h_2$ is proved in Lemma \ref{wdn2} in the Appendix.

As small time asymptotic is concerned, we approximate the two processes $\tilde{h}_1$, $\tilde{h}_2$ by their respective modal-path approximations $\hat h_1$, $\hat h_2$, resulting in the conditional expectation
\begin{equation}\label{expec-hat}
\tilde{\mathbb{E}}_{x,y} \left[e^{\int_0^{T} \hat{h}_1(t) d\tilde{W}_t - \frac12\int_0^{T} \hat{h}_1^2(t) dt + \int_0^{T} \hat{h}_2(t) d\tilde{B}_t - \frac12 \int_0^{T} \hat{h}_2^2(t) dt}\right].
\end{equation}
Now we can evaluate the conditional expectation \eqref{expec-hat} explicitly since the two random variables $G_1:=\int_0^{T} \hat{h}_1(t) d\tilde{W}_t$ and $G_2:=\int_0^{T} \hat h_2(t) d\tilde{B}_t$ are jointly Gaussian. Denote the integral of $h$ by $\intl{h}=\int_0^T h(t)dt$ for any $h\in L^1([0,T])$. 
\begin{lemma}\label{lem-ome}
The logarithm of \eqref{expec-hat} has the explicit expression
\bea\label{Jun-3-16}
\omega(T)&=& \log \tilde{\mathbb{E}}_{x,y} \left[e^{\int_0^{T} \hat{h}_1(t) d\tilde{W}_t - \frac12\int_0^{T} \hat{h}_1^2(t) dt + \int_0^{T} \hat{h}_2(t) d\tilde{B}_t - \frac12 \int_0^{T} \hat{h}_2^2(t) dt}\right] \nonumber\\
&=& \frac1{\bar\rho_H^2}\left\{ \left(\frac{\bar\rho \intl{\hat h_1}}{\sqrt T} + \frac{\rho \intl{\hat h_2}}{\sqrt T} - \frac{\rho_H\intl{\bar h_2}}{T^H}\right) \left(\frac{x - x_0}{\sqrt T}\right) \right.\nonumber\\
&&\left.\quad\quad-\rho_H \left(\frac{\bar\rho\intl{\hat h_1}}{\sqrt T} + \frac{\rho \intl{\hat h_2}}{\sqrt T} - \frac{\rho_H\intl{\bar h_2}}{T^H}\right) \left(\frac{y - y_0}{T^H}\right) + \bar\rho_H^2 \frac{\intl{\bar h_2}}{T^H} \left(\frac{y - y_0}{T^H} \right)\right. \nonumber \\
&& \left.  \quad\quad-\left(\frac{\bar\rho \intl{\hat h_1}}{\sqrt T} + \frac{\rho \intl{\hat h_2}}{\sqrt T} - \frac{\rho_H\intl{\bar h_2}}{T^H}\right)^2 -\left(\frac{\bar{\rho}_H\intl{\bar{h}_2}}{T^H}\right)^2\right\}. 
\eea
Furthermore, since the last two terms in the curly brackets on the right-hand side of \eqref{Jun-3-16} are of higher order compared to the others, we have as $T \to 0$
\bea\label{Jun-3-17}
\omega(T)&=& \log \tilde{\mathbb{E}}_{x,y} \left[e^{\int_0^{T} \hat{h}_1(t) d\tilde{W}_t - \frac12\int_0^{T} \hat{h}_1^2(t) dt + \int_0^{T} \hat{h}_2(t) d\tilde{B}_t - \frac12 \int_0^{T} \hat{h}_2^2(t) dt}\right] \nonumber\\
&=& \frac1{\bar\rho_H^2}\left\{ \left(\frac{\bar\rho \intl{\hat h_1}}{\sqrt T} + \frac{\rho \intl{\hat h_2}}{\sqrt T} - \frac{\rho_H\intl{\bar h_2}}{T^H}\right) \left(\frac{x - x_0}{\sqrt T}\right) \right. \nonumber\\
&& - \left. \rho_H \left(\frac{\bar\rho\intl{\hat h_1}}{\sqrt T} + \frac{\rho \intl{\hat h_2}}{\sqrt T} - \frac{\rho_H\intl{\bar h_2}}{T^H}\right) \left(\frac{y - y_0}{T^H}\right) + \bar\rho_H^2 \frac{\intl{\bar h_2}}{T^H} \left(\frac{y - y_0}{T^H} \right) \right\} + O(T^\alpha)\nonumber\\
&=&\omega_1(T)+ O(T^\alpha),
\eea
where 
\beaa
\omega_1(T)&=& \frac1{\bar\rho_H^2}\left\{ \left(\frac{\bar\rho \intl{\hat h_1}}{\sqrt T} + \frac{\rho \intl{\hat h_2}}{\sqrt T} - \frac{\rho_H\intl{\bar h_2}}{T^H}\right) \left(\frac{x - x_0}{\sqrt T}\right) \right. \nonumber\\
&& - \left. \rho_H \left(\frac{\bar\rho\intl{\hat h_1}}{\sqrt T} + \frac{\rho \intl{\hat h_2}}{\sqrt T} - \frac{\rho_H\intl{\bar h_2}}{T^H}\right) \left(\frac{y - y_0}{T^H}\right) + \bar\rho_H^2 \frac{\intl{\bar h_2}}{T^H} \left(\frac{y - y_0}{T^H} \right) \right\}
\eeaa
and
\begin{equation}\label{Jul-10}
\alpha =
\begin{cases}
3-4H,\ \mbox{if} \ H\in(\frac12,\frac34),\\
2H,\ \mbox{if} \ H\in(0,\frac12].
\end{cases}
\end{equation}
\end{lemma}

The following is the main theorem of the present paper. 
\begin{theorem}(Small time asymptotic of the joint density) \label{thm:small-time-exp} \\
Assume that $H<\frac34$. The joint probability density $p_T(x,y|x_0,y_0)$ given by \eqref{eqn:bridge-rep} in Theorem \ref{thm:bridge-rep} has the asymptotic expansion as $T \to 0$ 
\bea\label{conv-dist}
p_T(x,y|x_0,y_0) = \phi(x-x_0,y-y_0)e^{\omega_1(T)} ( 1+ o (T^\beta) ),
\eea
\if2
$ \phi\left(x-x_0, y-y_0\right) e^{\omega(T)}$ in the following sense: 
for any bounded and continuous function $f$ defined on $\mathbb{R}^2$ and arbitrary $ \epsilon > 0 $, 
\bea\label{conv-dist}
&&\lim\limits_{T\to 0} T^{\epsilon-\alpha} \int_{\mathbb{R}^2}  f(x,y)\left(p_T(x,y|x_0,y_0)-\phi(x-x_0,y-y_0)e^{\omega(T)}\right)dx dy=0,
\eea
where
\beaa
&& \omega(T) \\
&=& \frac1{\bar\rho_H^2}\left\{-\left(\frac{\bar\rho \intl{\hat h_1}}{\sqrt T} + \frac{\rho \intl{\hat h_2}}{\sqrt T} - \frac{\rho_H\intl{\bar h_2}}{T^H}\right)^2 -\left(\frac{\bar{\rho}_H\intl{\bar{h}_2}}{T^H}\right)^2\right.\nonumber\\
&&\left.+ \left(\frac{\bar\rho \intl{\hat h_1}}{\sqrt T} + \frac{\rho \intl{\hat h_2}}{\sqrt T} - \frac{\rho_H\intl{\bar h_2}}{T^H}\right) \left(\frac{x - x_0}{\sqrt T}\right) \right. \nonumber \\
&& - \left. \rho_H \left(\frac{\bar\rho\intl{\hat h_1}}{\sqrt T} + \frac{\rho \intl{\hat h_2}}{\sqrt T} - \frac{\rho_H\intl{\bar h_2}}{T^H}\right) \left(\frac{y - y_0}{T^H}\right) + \bar\rho_H^2 \frac{\intl{\bar h_2}}{T^H} \left(\frac{y - y_0}{T^H} \right) \right\},
\eeaa
and 
\fi
where $\beta\in(0,\alpha)$ ($\alpha$ is defined in \eqref{Jul-10}), $\phi$ is the Gaussian density given in \eqref{eqn:gaussian-density} and $\omega_1(T)$ is defined in \eqref{Jun-3-17}.
\end{theorem}
\if4
\begin{remark}
Note that the convergence of the asymptotic approximation in Theorem \ref{thm:small-time-exp} is in a sense of weak convergence.
\end{remark}
\fi
\begin{remark}(Classical heat kernel expansion) \\
Let $q(T,x_T,y_T|t,x_t,y_t)$ be the transition density of a two dimensional diffusion process from $(x_t,y_t)$ at time $t$ to $(x_T, y_T)$ at time $T$. Then as $t \to T$, $q$ has the following heat kernel expansion up to zeroth order
\bea
q_T(T,x_T,y_T|t,x_t, y_t) = \frac1{2\pi (T-t)} e^{-\frac{d^2}{2(T-t)}}\, e^{\int_\gamma \inn{V(\gamma(s)}{\dot\gamma(s)} ds}\, \left\{ 1 + O(T-t) \right\}, \label{eqn:hk-exp}
\eea
where $d$ denotes the geodesic distance between $(x_t,y_t)$ and $(x_T,y_T)$ associated with the Riemann metric determined by the diffusion matrix of the underlying process, assumed uniformly elliptic. $\int_\gamma \inn{V}{\dot\gamma} ds$ represents the work done by the vector field $V$, given by the drift of the underlying process, along the geodesic $\gamma$ connecting $(x_t,y_t)$ to $(x_T,y_T)$. See for instance Hsu \cite{eltonbook} (Theorem 5.1.1) for more detailed discussions on heat kernel expansion.
In the case of flat geometry (Euclidean), the geodesic is simply a straight line connecting the initial and terminal points and the geodesic distance is the Euclidean distance. When $H = \frac12$, we show in Example \ref{ex:hk} below that $\omega(T)$ in the approximation \eqref{conv-dist}can be expanded in line with heat kernel expansion up to zeroth order as follows
\beaa
&& \omega(T) \\
&=& \frac1{\bar\rho^2} \left\{\intl{\bar h_1} \left(\frac{x - x_0}{T}\right) - \rho \intl{\bar h_2} \left(\frac{x - x_0}{T}\right) -  \rho \intl{\bar h_1} \left(\frac{y - y_0}{T}\right) + \intl{\bar h_2} \left(\frac{y - y_0}T \right) \right\} + O(T).
\eeaa
as $T \to 0$, which recovers \eqref{eqn:hk-exp} in the Euclidean case. 
\end{remark}

\subsection{Proof of Lemma \ref{lem-ome}}$\mbox{ }$ \\
We prove Lemma \ref{lem-ome} in this subsection. 
\begin{proof}(\textbf{Proof of Lemma \ref{lem-ome}})

Consider the Gaussian random vector
$\bZ := (G_1, G_2, X_T, Y_T)^\prime$.
Note that $\bZ$ has expectation $(0,0,x_0, y_0)^\prime$ and covariance matrix $\mathbf{\Sigma}_{\bZ}$ 
\[
\mathbf{ \Sigma}_{\bZ} = \left[\begin{array}{cc}\mathbf{C}&\mathbf{D}\\ \mathbf{D}^\prime&\mathbf{\Sigma(T)}\end{array}\right],
\]
where
\[
\mathbf{C} = \left[\begin{array}{cc}
\intl{\hat h_1^2} & 0 \\
0 & \intl{\hat h_2^2}
\end{array}\right], \quad
\mathbf{D }= \left[\begin{array}{cc}
\bar\rho \intl{\hat h_1} & 0 \\
\rho \intl{\hat h_2} & \intl{\hat h_2K_H(T,\cdot) }=\intl{\bar h_2}
\end{array}\right],
\]
and $\mathbf{\Sigma}(T)$ is defined in \eqref{A_t}. Let $\1 = (1, 1)'$ denote the $2 \times 1$ column vector with both components being equal to 1. 

By applying Lemma \ref{conditional} in Appendix, we decompose the Gaussian vector $(G_1, G_2)^\prime$ as
\begin{eqnarray*}
\left[\begin{array}{c}G_1\\ G_2\end{array}\right]
&=&\mathbf{ D}\mathbf{\Sigma}(T)^{-1} \left[\begin{array}{c}X_T-x_0\\ Y_T-y_0\end{array}\right] + \left[\begin{array}{c}V_1\\ V_2\end{array}\right],
\end{eqnarray*}
where $(V_1, V_2)^\prime$ is a Gaussian vector independent of $(X_T, Y_T)^\prime$ with zero expectation and covariance matrix given by $\mathbf{C} - \mathbf{D}\mathbf{\Sigma}(T)^{-1}\mathbf{D}^\prime$.
Moreover, by straightforward computations, one can show that the matrix $\mathbf{D}\mathbf{\Sigma}(T)^{-1}\mathbf{D}^\prime$ has the following explicit expression
\beaa
\mathbf{D}\mathbf{\Sigma}(T)^{-1}\mathbf{D}^\prime= \frac1{\bar\rho_H^2}\left[\begin{array}{cc}
\left(\frac{\bar\rho \intl{\hat h_1}}{\sqrt T}\right)^2 & \frac{\bar\rho \intl{\hat h_1}}{\sqrt T} \, \frac{\rho \intl{\hat h_2}}{\sqrt T} - \frac{\bar\rho \intl{\hat h_1}}{\sqrt T}\, \frac{\rho_H \intl{\bar h_2}}{T^H} \\
\frac{\bar\rho \intl{\hat h_1}}{\sqrt T} \, \frac{\rho \intl{\hat h_2}}{\sqrt T} - \frac{\bar\rho \intl{\hat h_1}}{\sqrt T}\, \frac{\rho_H \intl{\bar h_2}}{T^H} & \left(\frac{\rho \intl{\hat h_2}}{\sqrt T}\right)^2 - 2 \frac{\rho \intl{\hat h_2}}{\sqrt T} \frac{\rho_H \intl{\bar h_2}}{T^H} + \left(\frac{\intl{\bar h_2}}{T^H}\right)^2
\end{array}\right].
\eeaa
Thus, we can calculate \eqref{expec-hat} by using the above decomposition
 \begin{eqnarray}\label{expo-xy}
 && \tilde{\mathbb{E}}_{x,y} \left[e^{\int_0^{T} \hat{h}_1(t) d\tilde{W}_t - \frac12\int_0^{T} \hat{h}_1^2(t) dt + \int_0^{T} \hat{h}_2(t) d\tilde{B}_t - \frac12 \int_0^{T} \hat{h}_2^2(t) dt}\right]\nonumber \\
 &=& e^{- \frac12\int_0^{T} \hat{h}_1^2(t) dt - \frac12\int_0^{T} \hat{h}_2^2(t) dt} \tilde{\mathbb{E}}_{x,y}\left[ e^{G_1 + G_2 } \right] \nonumber \\
&=& e^{ - \frac12 \left(\intl{\hat{h}_1^2} + \intl{\hat{h}_2^2}\right)} \times \exp\left\{\1'\,\mathbf{D}\mathbf{\Sigma}(T)^{-1} \left[\begin{array}{c}x-x_0\\ y-y_0\end{array}\right]\right\}     \times\tE{e^{V_1+V_2}}.
\end{eqnarray}
Note that
\bea
&&\1'\,\mathbf{D}\mathbf{\Sigma}(T)^{-1} \left[\begin{array}{c}x-x_0\\ y-y_0\end{array}\right]\nonumber\\
&=& \frac1{\bar\rho_H^2}\left\{
\left(\frac{\bar\rho \intl{\hat h_1}}{\sqrt T} + \frac{\rho \intl{\hat h_2}}{\sqrt T} - \frac{\rho_H\intl{\bar h_2}}{T^H}\right) \left(\frac{x - x_0}{\sqrt T}\right) \right. \nonumber\\
&& - \left. \rho_H \left(\frac{\bar\rho\intl{\hat h_1}}{\sqrt T} + \frac{\rho \intl{\hat h_2}}{\sqrt T} - \frac{\rho_H \intl{\bar h_2}}{T^H}\right) \left(\frac{y - y_0}{T^H}\right) + \bar\rho_H^2 \frac{\intl{\bar h_2}}{T^H} \left(\frac{y - y_0}{T^H} \right) \right\},
\eea
and
\begin{eqnarray} \label{expo-V} 
&& \tE{e^{V_1+V_2}} = e^{\frac12 \var(V_1 + V_2)} = e^{ \1' (\mathbf{C} - \mathbf{D}\mathbf{\Sigma}(T)^{-1}\mathbf{D}^\prime) \1 }\nonumber\\
&=& e^{\frac12 \left(\intl{\hat{h}_1^2} + \intl{\hat{h}_2^2}\right)} \, e^{ -\1' (\mathbf{D}\mathbf{\Sigma}(T)^{-1}\mathbf{D}^\prime) \1 }\nonumber\\
&=& e^{\frac12 \left(\intl{\hat{h}_1^2} + \intl{\hat{h}_2^2}\right)} 
\times 
\exp\left\{\, \frac{1}{\bar{\rho}_H^2}\left(\left(\frac{\bar{\rho}\intl{\hat{h}_1}}{\sqrt{T}}+\frac{\rho\intl{\hat{h}_2}}{\sqrt{T}}-\frac{\rho_H\intl{\bar{h}_2}}{T^H}\right)^2+\left(\frac{\bar{\rho}_H\intl{\bar{h}_2}}{T^H}\right)^2\right)\right\}. \nonumber\\
\end{eqnarray}
Therefore, by combining \eqref{expo-xy} - \eqref{expo-V} we obtain
\bea
&&\log \tilde{\mathbb{E}}_{x,y} \left[e^{\int_0^{T} \hat{h}_1(t) d\tilde{W}_t - \frac12\int_0^{T} \hat{h}_1^2(t) dt + \int_0^{T} \hat{h}_2(t) d\tilde{B}_t - \frac12 \int_0^{T} \hat{h}_2^2(t) dt}\right] \nonumber \\
&=&\1'\,\mathbf{D}\mathbf{\Sigma}(T)^{-1} \left[\begin{array}{c}x-x_0\\ y-y_0\end{array}\right]- \1' (\mathbf{D}\mathbf{\Sigma}(T)^{-1}\mathbf{D}^\prime) \1 \nonumber \\
&=& \frac1{\bar\rho_H^2}\left\{+ \left(\frac{\bar\rho \intl{\hat h_1}}{\sqrt T} + \frac{\rho \intl{\hat h_2}}{\sqrt T} - \frac{\rho_H\intl{\bar h_2}}{T^H}\right) \left(\frac{x - x_0}{\sqrt T}\right) \right.\nonumber\\
&&\left.\quad\quad-\rho_H \left(\frac{\bar\rho\intl{\hat h_1}}{\sqrt T} + \frac{\rho \intl{\hat h_2}}{\sqrt T} - \frac{\rho_H\intl{\bar h_2}}{T^H}\right) \left(\frac{y - y_0}{T^H}\right) + \bar\rho_H^2 \frac{\intl{\bar h_2}}{T^H} \left(\frac{y - y_0}{T^H} \right)\right. \nonumber \\
&& \left.  \quad\quad-\left(\frac{\bar\rho \intl{\hat h_1}}{\sqrt T} + \frac{\rho \intl{\hat h_2}}{\sqrt T} - \frac{\rho_H\intl{\bar h_2}}{T^H}\right)^2 -\left(\frac{\bar{\rho}_H\intl{\bar{h}_2}}{T^H}\right)^2\right\},\nonumber
\eea
which is the right-hand side of \eqref{Jun-3-16}.

Furthermore,  in the case of $H\in(\frac12,\frac34)$, by straightforward calculation based on \eqref{ap-2-51}, \eqref{Jun-hat-2} and \eqref{Jun-hat-1}, we obtain the following estimates
\bea\label{Jun-3-22}
\frac{\intl{\bar{h}_2}}{T^H}&\leq& C(1+|x-x_0|+|y-y_0|)T^{1-H}+C|y-y_0|T^{\frac{3-4H}{2}}\nonumber\\
&\leq&C(1+|x-x_0|+|y-y_0|)T^{\frac{3-4H}{2}},
\eea
and, for $i=1,2$,
\bea\label{Jun-3-23}
\frac{ \intl{\hat h_i}}{\sqrt T}&\leq&C(1+|x-x_0|+|y-y_0|)T^{1-H}+C|y-y_0|T^{\frac{3-4H}{2}}+CT^{1+\gamma-H}\nonumber\\
&&\quad +C|x-x_0|T^{1-H}+C|y-y_0|(T^{\frac{3-4H}{2}}+T^{1-H})\nonumber\\
&\leq& C(1+|x-x_0|+|y-y_0|)T^{\frac{3-4H}{2}}.
\eea
Then, \eqref{Jun-3-16}, \eqref{Jun-3-22} and \eqref{Jun-3-23} imply \eqref{Jun-3-17} for the case of $H\in(\frac12,\frac34)$.

In the case of $H\in(0, \frac12]$, it follows from \eqref{Jun-2-77} - \eqref{Jun-2-79} the inequalities
\bea\label{Jun-3-24}
\frac{\intl{\bar{h}_2}}{T^H}&\leq& C(1+|x-x_0|+|y-y_0|)T^{1-H}+C|x-x_0|T^{\frac12}\nonumber\\
&\leq&C(1+|x-x_0|+|y-y_0|)T^{\frac12}\nonumber\\
&\leq& C(1+|x-x_0|+|y-y_0|)T^{H}.
\eea
and, for $i=1,2$,
\bea\label{Jun-3-25}
\frac{ \intl{\hat h_i}}{\sqrt T}&\leq&C(1+|x-x_0|+|y-y_0|)T^{\frac12}+C|x-x_0|T^{H}\nonumber\\
&\leq& C(1+|x-x_0|+|y-y_0|)T^{H}.
\eea
Therefore, by \eqref{Jun-3-16}, \eqref{Jun-3-24} and \eqref{Jun-3-25} we obtain \eqref{Jun-3-17} for the case of $H\in(0,\frac12]$.
\end{proof}

\subsection{Proof of Theorem \ref{thm:small-time-exp}} $\mbox{}$

This subsection is devoted to the proof of Theorem \ref{thm:small-time-exp}.
\begin{proof}(\textbf{Proof of Theorem \ref{thm:small-time-exp}})

 For any fixed $(x,y)\in\mathbb{R}^2$, we denote
\[
\xi_T(x,y)= e^{\int_0^{T} \hat{h}_1(t) d\tilde{W}_t - \frac12\int_0^{T} \hat{h}_1^2(t) dt + \int_0^{T} \hat{h}_2(t) d\tilde{B}_t - \frac12 \int_0^{T} \hat{h}_2^2(t) dt},
\]
\[
\xi_T= e^{\int_0^{T} \th_1(t) d\tilde{W}_t - \frac12\int_0^{T} \th_1^2(t) dt + \int_0^{T} \th_2(t) d\tilde{B}_t - \frac12 \int_0^{T} \th_2^2(t) dt},
\]
and
\begin{eqnarray}\label{eta-T}
\eta_T(x,y)&=&\int_0^T(\th_1(t)-\hat{h}_1(t)) d\tilde{W}_t - \frac12\int_0^{T}(\th_1(t)^2- \hat{h}_1^2(t)) dt\nonumber\\
&&+\int_0^T(\th_2(t)-\hat{h}_2(t)) d\tilde{B}_t - \frac12\int_0^{T}(\th_2(t)^2- \hat{h}_2^2(t)) dt.
\end{eqnarray}

This theorem claims that $\mathbb{E}_{x,y}\left[\xi_T(x,y)\right]$ is a limit of $\mathbb{E}_{x,y}\left[\xi_T\right]$ as $T\to 0$. From Theorem \ref{thm:bridge-rep} and the calculation of $\omega(T)$ in Lemma \ref{lem-ome} we only need to show that, for any bounded and continuous function $f$ defined on $\mathbb{R}^2$, the following limit holds 
\bea\label{conv-dist-1}
&&\lim\limits_{T\to 0} \int_{\mathbb{R}^2}\tilde p_T(x,y|x_0, y_0) f(x,y)\left( \tilde{\mathbb E}_{x,y}\left[\xi_T\right]-\tilde{\mathbb{E}}_{x,y}\left[\xi_T(x,y)\right]\right)dx dy\nonumber\\
&=&\lim\limits_{T\to 0} \tE{f(X_T,Y_T)(\xi_T-\xi_T(x,y))}=0.
\eea
Notice that 
\bea
\lim\limits_{T\to 0} \left|\tE{f(X_T,Y_T)(\xi_T-\xi_T(x,y))}\right|\leq \Vert f\Vert_\infty\lim\limits_{T\to 0}\tE{\left|\xi_T-\xi_T(x,y)\right|}.
\eea
Thus, in order to show \eqref{conv-dist-1}, it suffices to show
\begin{equation}\label{lim-diff}
\lim\limits_{T\to 0}\tE{\left|\xi_T-\xi_T(x,y)\right|}=0.
\end{equation}
Using the inequality $|e^u-e^v|\leq \frac{e^u+e^v}{2} |u-v|$ for any $u,v\in\mathbb{R}$, we have the following bound 
\begin{eqnarray}\label{lim-2-D}
|\xi_T-\xi_T(x,y)|
&\leq&\frac12\left(\xi_T+\xi_T(x,y)\right)|\eta_T(x,y)|\nonumber\\ 
&=&\frac12\tilde{D}_T(x,y)+\frac12\hat{D}_T(x,y),
\end{eqnarray}
where
\[
\tilde{D}_T(x,y)=\xi_T|\eta_T(x,y)|=e^{\int_0^{T} \tilde{h}_1(t) d\tilde{W}_t - \frac12\int_0^{T} \tilde{h}_1^2(t) dt + \int_0^{T} \tilde{h}_2(t) d\tilde{B}_t - \frac12 \int_0^{T} \tilde{h}_2^2(t) dt}|\eta_T(x,y)|,
\]
and
\[
\hat{D}_T(x,y)=\xi_T(x,y)|\eta_T(x,y)|=e^{\int_0^{T} \hat{h}_1(t) d\tilde{W}_t - \frac12\int_0^{T} \hat{h}_1^2(t) dt + \int_0^{T} \hat{h}_2(t) d\tilde{B}_t - \frac12 \int_0^{T} \hat{h}_2^2(t) dt}|\eta_T(x,y)|.
\]

First, for any  $r>0$, we will estimate $\tE{|\eta_T(x,y)|^r}$.  From \eqref{eta-T}, the Burkholder-Davis-Gundy inequality and the Cauchy-Schwartz inequality, it implies that
\begin{eqnarray}\label{ap-2-69}
&&\tilde{\mathbb{E}}\left[|\eta_T(x,y)|^r\right]\nonumber\\
&\leq&C\left(\tilde{\mathbb{E}}\left[\left|\int_0^T(\tilde{h}_1(t)-\hat{h}_1(t))d\tilde{W}_t\right|^r\right]+\tilde{\mathbb{E}}\left[\left|\int_0^T(\tilde{h}_2(t)-\hat{h}_2(t))d\tilde{B}_t\right|^r\right]\right.\nonumber\\
&&+\tilde{\mathbb{E}}\left[\left(\int_0^T(|\tilde{h}_1(t)|^2+|\hat{h}_1(t)|^2)dt\right)^{\frac{r}2}\left(\int_0^T(|\tilde{h}_1(t)-\hat{h}_1(t)|^2)dt\right)^{\frac{r}{2}}\right]\nonumber\\
&&+\tilde{\mathbb{E}}\left[\left(\int_0^T(|\tilde{h}_2(t)|^2+|\hat{h}_2(t)|^2)dt\right)^{\frac{r}2}\left(\int_0^T(|\tilde{h}_2(t)-\hat{h}_2(t)|^2)dt\right)^{\frac{r}{2}}\right]\nonumber\\
&\leq&\tilde{\mathbb{E}}\left[\left(\int_0^T|\tilde{h}_1(t)-\hat{h}_1(t)|^2dt\right)^{\frac{r}2}\right]+\tilde{\mathbb{E}}\left[\left(\int_0^T|\tilde{h}_2(t)-\hat{h}_2(t)|^2dt\right)^{\frac{r}2}\right]\nonumber\\
&&+\tilde{\mathbb{E}}\left[\left(\int_0^T(|\tilde{h}_1(t)|^2+|\hat{h}_1(t)|^2)dt\right)^{\frac{r}2}\left(\int_0^T(|\tilde{h}_1(t)-\hat{h}_1(t)|^2)dt\right)^{\frac{r}{2}}\right]\nonumber\\
&&+\tilde{\mathbb{E}}\left[\left(\int_0^T(|\tilde{h}_2(t)|^2+|\hat{h}_2(t)|^2)dt\right)^{\frac{r}2}\left(\int_0^T(|\tilde{h}_2(t)-\hat{h}_2(t)|^2)dt\right)^{\frac{r}{2}}\right].\nonumber\\
\end{eqnarray}
From Lemma \ref{lem-7} in Appendix, we can bound the right-hand side of \eqref{ap-2-69} in the case of $H\leq\frac12$ and the case of $H\in (\frac12,\frac34)$.

When $H\leq\frac12$, the right-hand side of \eqref{ap-2-69} can be bounded by
\bea\label{Jun-2-90}
&&C(1+\tE{\Vert\tilde{B}\Vert_\infty^r+\Vert\tilde{W}\Vert_\infty^r+\Vert\tilde{B^H}\Vert_\infty^r}+|x-x_0|^r+|y-y_0|^r)T^{Hr}\nonumber\\
&+&C(1+\tE{\Vert\tilde{B}\Vert_\infty^{2r}+\Vert\tilde{W}\Vert_\infty^{2r}+\Vert\tilde{B^H}\Vert_\infty^{2r}}+|x-x_0|^{2r}+|y-y_0|^{2r})T^{2Hr}\nonumber\\
&\to& 0,\ \mbox{as}\ T\to 0.
\eea
When $H\in (\frac12,\frac34)$, the right-hand side of \eqref{ap-2-69} can be bounded by
\bea\label{Jun-2-91}
&&C(1+\tE{\Vert\tilde{B}\Vert_{\frac12-\epsilon}^r+\Vert\tilde{W}\Vert^r_{\frac12-\epsilon}+\Vert\tilde{B}^H\Vert^r_{\frac12-\epsilon}+\Vert \tilde{B}\Vert^r_{\infty}+\Vert \tilde{W}\Vert^r_{\infty}+\Vert \tilde{B}^H\Vert^r_{\infty}})T^{(1-H)r}\nonumber\\
&+&C(|x-x_0|^r+|y-y_0|^r)T^{(1-H)r}+C|y-y_0|^rT^{\frac{(3-4H)r}{2}}\nonumber\\
&+&C(1+\tE{\Vert\tilde{B}\Vert_{\frac12-\epsilon}^{2r}+\Vert\tilde{W}\Vert^{2r}_{\frac12-\epsilon}+\Vert\tilde{B}^H\Vert^{2r}_{\frac12-\epsilon}+\Vert \tilde{B}\Vert^{2r}_{\infty}+\Vert \tilde{W}\Vert^{2r}_{\infty}+\Vert \tilde{B}^H\Vert^{2r}_{\infty}})T^{(2-2H)r}\nonumber\\
&+&C(|x-x_0|^{2r}+|y-y_0|^{2r})T^{(2-2H)r}+C|y-y_0|^{2r}T^{(3-4H)r}\nonumber\\
&\to& 0,\ \mbox{as}\ T\to 0.
\eea
Therefore, by \eqref{ap-2-69} - \eqref{Jun-2-91} we can show 
\begin{equation}\label{lim-eta}
\lim\limits_{T\to0}\tilde{\mathbb{E}}\left[|\eta_T(x,y)|^r\right]=0,
\end{equation}
for any $r>0$.

Now, we fix $p>0$ and $q>0$ with $\frac1p+\frac1q=1$. First, let us show $\lim\limits_{T\to 0}\tilde{\mathbb{E}}\left[\hat{D}_T(x,y)\right]=0$. Using H\"{o}lder's inequality for conditional expectation, by \eqref{lim-eta}, Part (b) for the case of $H\in(\frac12,\frac34)$ and Part (f) for the case of $H\leq \frac12$ in Lemma \ref{lem-7} we obtain
\begin{eqnarray}\label{hat-D}
&&\lim\limits_{T\to 0}\tilde{\mathbb{E}}\left[\hat{D}_T(x,y)\right]\nonumber\\
&\leq &\lim\limits_{T\to 0}\left(\tilde{\mathbb{E}}\left[e^{\int_0^{T} p\hat{h}_1(t) d\tilde{W}_t - \frac{p}2\int_0^{T} \hat{h}_1^2(t) dt + \int_0^{T}p \hat{h}_2(t) d\tilde{B}_t - \frac{p}2 \int_0^{T} \hat{h}_2^2(t) dt}\right]\right)^{\frac1p}\left(\tilde{\mathbb{E}}\left[|\eta_T(x,y)|^q\right]\right)^{\frac1q}\nonumber\\
&=&\left(\lim\limits_{T\to 0}e^{\frac{p(p-1)}2\int_0^{T}( \hat{h}_1^2(t) + \hat{h}_2^2(t)) dt}\right)\left(\lim\limits_{T\to 0}\left(\tilde{\mathbb{E}}\left[|\eta_T(x,y)|^q\right]\right)^{\frac1q}\right)=1\cdot 0=0.
\end{eqnarray}

Next, we will show $\lim\limits_{T\to 0}\tE{\tilde{D}_T(x,y)}=0$. Using the techniques in the proof of Lemma \ref{lma:novikov} and the estimates in Lemma \ref{lem-7}, we can find some $0<t_1<1$ such that
\[
\tilde{\mathbb{E}}\left[ e^{\frac{p^2}2\int_0^{t_1} \tilde{h}_1^2(t) dt+\frac{p^2}2\int_0^{t_1} \tilde{h}_1^2(t) dt}\right]<\infty,
\]
and hence
\begin{equation}\label{exp-1}
\tilde{\mathbb{E}}\left[e^{\int_0^{T} p\tilde{h}_1(t) d\tilde{W}_t - \frac{p^2}2\int_0^{T} \tilde{h}_1^2(t) dt + \int_0^{T} p\tilde{h}_2(t) d\tilde{B}_t - \frac{p^2}2 \int_0^{T} \tilde{h}_2^2(t) dt}\right]=1.
\end{equation}
Without loss of generality, we assume $T<t_1$. Applying H\"{o}lder's inequality, the Cauchy-Schwartz inequality, Part (a) for the case of $H\in(\frac12,\frac34)$ and Part (e) for the case of $H\leq \frac12$ in Lemma \ref{lem-7},  \eqref{lim-eta}, \eqref{exp-1} and the dominated convergence theorem, one can obtain
\begin{eqnarray}\label{til-D}
&&\lim\limits_{T\to 0}\tilde{\mathbb{E}}\left[\tilde{D}_T(x,y)\right]\nonumber\\
&=&\lim\limits_{T\to 0}\tilde{\mathbb{E}}\left[e^{\int_0^{T} \tilde{h}_1(t) d\tilde{W}_t - \frac12\int_0^{T} \tilde{h}_1^2(t) dt + \int_0^{T} \tilde{h}_2(t) d\tilde{B}_t - \frac12 \int_0^{T} \tilde{h}_2^2(t) dt}|\eta_T(x,y)|\right]\nonumber\\
&\leq&\lim\limits_{T\to 0}\left[\left(\tilde{\mathbb{E}}\left[e^{\int_0^{T} p\tilde{h}_1(t) d\tilde{W}_t - \frac{p^2}2\int_0^{T} \tilde{h}_1^2(t) dt + \int_0^{T} p\tilde{h}_2(t) d\tilde{B}_t - \frac{p^2}2 \int_0^{T} \tilde{h}_2^2(t) dt}\right]\right)^{\frac1p}\right.\nonumber\\
&&\left.\times\left(\tilde{\mathbb{E}}\left[e^{(p-1)q\int_0^{T} \tilde{h}_1^2(t) dt+(p-1)q\int_0^{T} \tilde{h}_2^2(t) dt}\right]\right)^{\frac1{2q}}\left(\tilde{\mathbb{E}}\left[|\eta_T(x,y)|^{2q}\right]\right)^{\frac1{2q}}\right]\nonumber\\
&=&\lim\limits_{T\to 0}\left(\tilde{\mathbb{E}}\left[e^{(p-1)q\int_0^{T} \tilde{h}_1^2(t) dt+(p-1)q\int_0^{T} \tilde{h}_2^2(t) dt}\right]\right)^{\frac1{2q}}\lim\limits_{T\to 0}\left(\tilde{\mathbb{E}}\left[|\eta_T(x,y)|^{2q}\right]\right)^{\frac1{2q}}\nonumber\\
&\leq&1\cdot 0=0.
\end{eqnarray}
Therefore, \eqref{lim-2-D}, \eqref{hat-D} and \eqref{til-D} imply \eqref{lim-diff}, which completes the proof. 
\end{proof}

\subsection{Examples} $\mbox{ }$ \label{sec:examples} \\
We illustrate the $\omega(T)$ in the modal-path approximation \eqref{conv-dist} more explicitly by considering the following particular examples. Example \ref{ex:hk} shows the recovery of the classical heat kernel expansion when $H = \frac12$. 
\begin{example}
If the drift terms $h_1$ and $h_2$ are both independent of $x$ and $y$, we have $\tilde{h}_i=\hat{h}_i$, for $i=1, 2$. One can easily verify that the representation
\beaa
p_T(x,y|x_0,y_0)=\phi(x-x_0, y-y_0)e^{\omega(T)}
\eeaa
as in \eqref{conv-dist} is exact. 
\end{example}

\begin{example}(Recovery of classical heat kernel expansion up to zeroth order) \label{ex:hk} \\
Let $H = \frac12$. Note that in this case $\rho_H = \rho$, $\hat h_2 = \bar h_2$. Thus, $\bar\rho \hat h_1 = \bar h_1 - \rho \bar h_2$. The function $\omega$ simplifies to
\beaa
&& \omega(T) \\
&=& \frac1{\bar\rho^2} \left\{-\left(\frac{\intl{\bar h_1}-\rho\intl{\bar h_2}}{\sqrt T} \right)^2 -\left(\frac{\bar \rho \intl{\bar h_2}}{\sqrt{T}}\right)^2+ \left(\intl{\bar h_1} - \rho \intl{\bar h_2} \right) \left(\frac{x - x_0}{T}\right) \right. \\
&& + \left.  (\intl{\bar h_2} - \rho\intl{\bar h_1}) \left(\frac{y - y_0}{T}\right)  \right\} \\
&=& \frac1{\bar\rho^2} \left\{\intl{\bar h_1} \left(\frac{x - x_0}{T}\right) - \rho \intl{\bar h_2} \left(\frac{x - x_0}{T}\right) -  \rho \intl{\bar h_1} \left(\frac{y - y_0}{T}\right) + \intl{\bar h_2} \left(\frac{y - y_0}T \right) \right\} + O(T).
\eeaa
Notice that the last expression is exactly the work done by the vector field $h_1 \p_1 + h_2 \p_2$ along the geodesic connecting $(x_0, y_0)$ to $(x,y)$. In this case, the geodesic is simply the straight line connecting $(x_0, y_0)$ and $(x,y)$.
Hence, the small time approximation of $p_T(x,y|x_0,y_0) $ reads 
\bea\label{heat-Jun}
\phi\left(x-x_0, y-y_0\right) e^{ \frac1{\bar\rho^2} \left\{\intl{\bar h_1} \left(\frac{x - x_0}{T}\right) - \rho \intl{\bar h_2} \left(\frac{x - x_0}{T}\right) -  \rho \intl{\bar h_1} \left(\frac{y - y_0}{T}\right) + \intl{\bar h_2} \left(\frac{y - y_0}T \right) \right\}} \left\{ 1 + O(T) \right\} 
\eea
as $T \to 0$. It recovers the classical heat kernel expansion to zeroth order in the two dimensional Euclidean case. 
\end{example}

\begin{example}(Uncorrelated case, i.e., $\rho = 0$) \\
In this case, since $\rho_H = \rho \kappa_H = 0$, $\bar \rho_H = \sqrt{1 - \rho_H^2} = 1$, and notice that $\hat h_1 = \bar h_1$, $\omega$ reduces to
\beaa
\omega(T) = -\left(\frac{\intl{\bar h_1}}{\sqrt T} \right)^2-\left(\frac{\intl{\bar h_2}}{ T^H} \right)^2 + \intl{\bar h_1} \left(\frac{x - x_0}T\right) + \intl{\bar h_2} \left(\frac{y - y_0}{T^{2H}}\right).
\eeaa
Hence,
\beaa
e^{\omega(T)} = e^{\intl{\bar h_1} \left(\frac{x - x_0}T\right) + \intl{\bar h_2} \left(\frac{y - y_0}{T^{2H}}\right)}[1 + o(T^\alpha)].
\eeaa
Thus, the small time approximation of $p_T(x,y|x_0,y_0) $ in this case reads 
\beaa
 \phi(x-x_0, y-y_0) e^{\intl{\bar h_1} \left(\frac{x - x_0}T\right) + \intl{\bar h_2} \left(\frac{y - y_0}{T^{2H}}\right)} (1 + o(T^\alpha)),
\eeaa
as $T\to 0$, which can be regarded as a generalization of the heat kernel expansion up to zeroth order, c.f. \eqref{heat-Jun} in Example \ref{ex:hk}.
\end{example}
\begin{example}
Consider the case where both $h_1$ and $h_2$ are linear functions of $x$ and $y$, say,
\beaa
&& h_i(t,x,y) = \alpha_i(t) x + \beta_i(t) y + \gamma_i(t), \ i = 1, 2.
\eeaa
We impose the following conditions:
\begin{itemize}
\item[(a1)] If $H>\frac12$, we assume that $\alpha_2(0)=\beta_2(0)=0$ and there exists two constants $L>0$ and $\gamma\in(H-\frac12,\frac12)$ such that 
\begin{equation*}
|\alpha_1(t)|+|\beta_1(t)|+|\gamma_1(t)|\leq L,\ \forall t\in[0,T],
\end{equation*}
and
\begin{equation*}
|\alpha_2(t)-\alpha_2(s)|+|\beta_2(t)-\beta_2(s)|+|\gamma_2(t)-\gamma_2(s)|\leq L|t-s|^\gamma,\ \forall s, t\in[0,T].
\end{equation*}
\item[(a2)] If $H\leq \frac12$, we assume that there exists a constant $L>0$ such that
\begin{equation*}
|\alpha_i(t)|+|\beta_i(t)|+|\gamma_i(t)|\leq L,\ \forall t\in[0,T], \ i=1,2.
\end{equation*}
\end{itemize}
Note that the above conditions ensure \eqref{Lip} and \eqref{linear}, and hence Theorem \ref{holder} stays true and \eqref{e-2-4}, \eqref{Jan-2-5} and \eqref{Jan-2-6} still hold.
Moreover, though \eqref{Hold} cannot be guaranteed in this example, we have the following estimate, for any small enough $\epsilon$,
\bea\label{Jun-3-38}
&&|h_2(t,X_t,Y_t)-h_2(s,X_s,Y_s)|\nonumber\\
&\leq& |\gamma_2(t)-\gamma_2(s)|+|\alpha_2(t)-\alpha_2(s)||X_t|+|\alpha_2(s)||X_t-X_s|\nonumber\\
&&\ +|\beta_2(t)-\beta_2(s)||Y_t|+|\beta_2(s)||Y_t-Y_s|\nonumber\\
&\leq&C(1+\Vert B\Vert_\infty+\Vert W\Vert_\infty+\Vert B^H\Vert_\infty)|t-s|^\gamma\nonumber\\
&&\ +C(1+\Vert B\Vert_\infty+\Vert W\Vert_\infty+\Vert B^H\Vert_\infty+\Vert B\Vert_{\frac12-\epsilon}+\Vert W\Vert_{\frac12-\epsilon})|t-s|^{\frac12-\epsilon}\nonumber\\
&&\ +C(1+\Vert B\Vert_\infty+\Vert W\Vert_\infty+\Vert B^H\Vert_\infty+\Vert B^H\Vert_{H-\epsilon})|t-s|^{H-\epsilon}.
\eea
Hence, by Lemma \ref{inverse}, the $\tilde{h}_i$'s are well defined. Using the above estimate and modifying the proof slightly in Lemma \ref{lma:novikov}, it follows that the Novikov's condition in Lemma \ref{lma:novikov} and the change of measure in Lemma \ref{Girsanov} sustain, thereby all the main results in this paper hold for this linear system. 
More importantly, the restriction $H < \frac34$ in Theorem \ref{thm:small-time-exp} can be removed 
in this linear case. In fact, using \eqref{est-m11}-\eqref{est-m22}, \eqref{Jun-3-38} and the assumptions in this example, the following estimates hold.
\begin{enumerate}
\item 
\beaa
|h_1(t,x_t^{x,y}, y_t^{x,y})|
&\leq& K(1+|x_t^{x,y}|+|y_t^{x,y}|)\\
&\leq&C(1+|x-x_0|+|y-y_0|+T^{H-\frac12}|x-x_0|+T^{\frac12-H}|y-y_0|),
\eeaa
which implies
\[
\int_0^T|h_1(t,x_t^{x,y},y_t^{x,y})|^2dt\leq C(1+|x-x_0|+|y-y_0|)T^{\beta},
\]
where
\begin{equation*}
\beta=\begin{cases}
2H,\ \mbox{if}\ H\leq \frac12,\\
2-2H, \ \mbox{if}\ H>\frac12;
\end{cases}
\end{equation*}
\item when $H\leq \frac12$,
\beaa
|h_2(t,x_t^{x,y}, y_t^{x,y})|
&\leq& K(1+|x_t^{x,y}|+|y_t^{x,y}|)\\
&\leq&C(1+|x-x_0|+|y-y_0|+T^{H-\frac12}|x-x_0|),
\eeaa
and when $H>\frac12$,
\beaa 
&&|h_2(t,x_t^{x,y}, y_t^{x,y})|\\
&\leq& |\alpha_2(t)x_0+\beta_2(t)y_0+\gamma_2(t)|+|\alpha_2(t)m_{11}(t;T)+\beta_2(t)m_{21}(t;T)||x-x_0|\\
&&+|\alpha_2(t)m_{12}(t;T)+\beta_2(t)m_{22}(t;T)||y-y_0|\\
&\leq& C(1+|x-x_0|+|y-y_0|+T^{\frac12-H}t^\gamma|y-y_0|);
\eeaa
\item when $H>\frac12$,
\beaa
&&|h_2(t,x_t^{x,y}, y_t^{x,y})-h_2(s,x_s^{x,y}, y_s^{x,y})|\\
&\leq&|(\alpha_2(t)-\alpha_2(s))x_0+(\beta_2(t)-\beta_2(s))y_0+(\gamma_2(t)-\gamma_2(s))|\\
&&+\big[|\alpha_2(t)-\alpha_2(s)||m_{11}(t;T)|+|\alpha_2(s)||m_{11}(t;T)-m_{11}(s;T)|\\
&&+|\beta_2(t)-\beta_2(s)||m_{21}(t;T)|+|\beta_2(s)||m_{21}(t;T)-m_{21}(s;T)|\big]|x-x_0|\\
&&+\big[|\alpha_2(t)-\alpha_2(s)||m_{12}(t;T)|+|\alpha_2(s)||m_{12}(t;T)-m_{12}(s;T)|\\
&&+|\beta_2(t)-\beta_2(s)||m_{22}(t;T)|+|\beta_2(s)||m_{22}(t;T)-m_{12}(s;T)|\big]|y-y_0|\\
&\leq&C|t-s|^\gamma+C\left(|t-s|^\gamma+\frac{|t-s|^{\frac12}}{T^{\frac12}}\right)|x-x_0|\\
&&+C\left(T^{\frac12-H}|t-s|^\gamma+\frac{s^\gamma|t-s|^{\frac12}}{T^H}+|t-s|^\gamma+\frac{|t-s|^H}{T^H}\right)|y-y_0|.
\eeaa
\end{enumerate}
Thus, analogue to the proofs of Lemma \ref{lem-7} one can show the bounds (e), (f), (g) in Lemma \ref{lem-7} in the case of $H\leq \frac12$. For $H\in(\frac12,1)$ one has
\begin{itemize}
\item[(4)] 
\beaa
|h_i(t,X_t, Y_t)|
&\leq& K(1+|X_t|+|Y_t|)\\
&\leq&C(1+\Vert \tilde{B}\Vert_\infty+\Vert \tilde{W}\Vert_\infty+\Vert \tilde{B}^H\Vert_\infty),\ i=1, 2;
\eeaa
\item[(5)] for small enough $\epsilon$,
\beaa
&&|h_2(t,X_t,Y_t)-h_2(s,X_s,Y_s)|\nonumber\\
&\leq& |\gamma_2(t)-\gamma_2(s)|+|\alpha_2(t)-\alpha_2(s)||X_t|+|\alpha_2(s)||X_t-X_s|\nonumber\\
&&\ +|\beta_2(t)-\beta_2(s)||Y_t|+|\beta_2(s)||Y_t-Y_s|\nonumber\\
&\leq&C(1+\Vert \tilde{B}\Vert_\infty+\Vert \tilde{W}\Vert_\infty+\Vert \tilde{B}^H\Vert_\infty)|t-s|^\gamma\\
&&+C(\Vert \tilde{B}\Vert_{\frac12-\epsilon}+\Vert \tilde{W}\Vert_{\frac12-\epsilon}+\Vert \tilde{B}^H\Vert_{\frac12-\epsilon})|t-s|^{\frac12-\epsilon}.
\eeaa
\end{itemize}
Similar to the proofs in Lemma \ref{lem-7}, using (2)-(5) we can show (a) in Lemma \ref{lem-7} and the following two inequalities
\beaa
&&\int_0^T(|\hat{h}_t(t)|^2+|\hat{h}_2(t)|^2)dt\\
&\leq& C(1+|x-x_0|^2+|y-y_0|^2)T^{2-2H}+C|y-y_0|^2T^{2\gamma+3-4H}\\
&&+C(1+|x-x_0|^2+|y-y_0|^2)T^{2\gamma+2-2H}+C|y-y_0|^2T^{2\gamma+3-4H}\\
&\leq&C(1+|x-x_0|^2+|y-y_0|^2)T^{2-2H},
\eeaa
since $\gamma>H-\frac12>0$ (and hence $2\gamma+2-4H>2-2H$), and
\beaa
&&\int_0^T(|\tilde{h}_1(t)-\hat{h}_1(t)|^2+|\tilde{h}_2(t)-\hat{h}_2(t)|^2)dt\nonumber\\
&\leq&C(1+\Vert\tilde{B}\Vert_{\frac12-\epsilon}^2+\Vert\tilde{W}\Vert^2_{\frac12-\epsilon}+\Vert\tilde{B}^H\Vert^2_{\frac12-\epsilon}+\Vert \tilde{B}\Vert^2_{\infty}+\Vert \tilde{W}\Vert^2_{\infty}+\Vert \tilde{B}^H\Vert^2_{\infty})T^{2-2H}\nonumber\\
&&+C(|x-x_0|^2+|y-y_0|^2)T^{2-2H}.
\eeaa

Therefore, in this linear case Theorem \ref{thm:small-time-exp} holds without the restriction $H<\frac34$. Furthermore, based on the above estimates (1)-(3), we can reduce $\omega(T)$ to be
\beaa
\omega(T)
&=& \frac1{\bar\rho_H^2}\left\{ \left(\frac{\bar\rho \intl{\hat h_1}}{\sqrt T} + \frac{\rho \intl{\hat h_2}}{\sqrt T} - \frac{\rho_H\intl{\bar h_2}}{T^H}\right) \left(\frac{x - x_0}{\sqrt T}\right) \right. \nonumber\\
&& - \left. \rho_H \left(\frac{\bar\rho\intl{\hat h_1}}{\sqrt T} + \frac{\rho \intl{\hat h_2}}{\sqrt T} - \frac{\rho_H\intl{\bar h_2}}{T^H}\right) \left(\frac{y - y_0}{T^H}\right) + \bar\rho_H^2 \frac{\intl{\bar h_2}}{T^H} \left(\frac{y - y_0}{T^H} \right) \right\} + O(T^\alpha),\nonumber\\
\eeaa
where 
\begin{equation*}
\alpha =
\begin{cases}
2-2H,\ \mbox{if} \ H\in(\frac12,\frac34),\\
2H,\ \mbox{if} \ H\in(0,\frac12].
\end{cases}
\end{equation*}
\end{example}

%
%

\section{Appendix} \label{sec:prelim}
In this appendix, after reviewing basic but essential background technicalities for dealing with the fractional Brownian motion, several lemmas that are used in the main part will be established.

\subsection{Fractional integrals and derivatives}\label{fid} $\mbox{}$

Let $a,b\in \mathbb{R}$ with $a<b$. Denote by $L^p([a,b]),\,p\geq1$, the usual space of Lebesgue
measurable functions $f:\,[a,b]\rightarrow\mathbb{R}$ for which $\Vert f\Vert_{L^p}<\infty$, where
\begin{equation*}
\Vert f\Vert_{L^p}=
\begin{cases}
\left(\int_a^b\vert f(t)\vert^p dt\right)^{1/p},\,\text{if $1\leq p<\infty$}\\
\text{ess} \sup\{\vert f(t)\vert:\,t\in[a,b]\},\,\text{if $p=\infty$}.
\end{cases}
\end{equation*}
Let $f\in L^{1}\left( [a,b]\right) $ and $%
\alpha >0.$ The left-sided and right-sided fractional Riemann-Liouville
integrals of $f$ of order $\alpha $ are defined for almost all $t\in \left(
a,b\right) $ by
\[
I_{a+}^{\alpha }f\left( t\right) =\frac{1}{\Gamma \left( \alpha \right) }%
\int_{a}^{t}\left( t-s\right) ^{\alpha -1}f\left( s\right) ds
\]%
and
\[
I_{b-}^{\alpha }f\left( t\right) =\frac{\left( -1\right) ^{-\alpha }}{\Gamma
\left( \alpha \right) }\int_{t}^{b}\left( s-t\right) ^{\alpha -1}f\left(
s\right) ds,
\]%
respectively, where $\left( -1\right) ^{-\alpha }=e^{-i\pi \alpha }$ and $%
\Gamma \left( \alpha \right) =\int_{0}^{\infty }r^{\alpha -1}e^{-r}dr$ is
the Euler gamma function. Let $I_{a+}^{\alpha }(L^{p}([a,b]))$ (resp. $%
I_{b-}^{\alpha }(L^{p}([a,b]))$) be the image of $L^{p}([a,b])$ by the operator $%
I_{a+}^{\alpha }$ (resp. $I_{b-}^{\alpha }$).

Fractional integration admits the following composition formulas:
\begin{equation}\label{comp-1}
I_{a+}^{\alpha}I^\beta_{a+}f=I^{\alpha+\beta}_{a+}f,
\end{equation}
and
\begin{equation}
I_{b-}^{\alpha}I^\beta_{b-}f=I^{\alpha+\beta}_{b-}f,
\end{equation}
for any $f\in L^1([a,b])$.

If $f\in I_{a+}^{\alpha
}\left( L^{p}([a,b])\right) \ $ (resp. $f\in I_{b-}^{\alpha }\left( L^{p}([a,b])\right) $)
and $0<\alpha <1$ then the Weyl derivatives are defined as
\begin{equation*}
D_{a+}^{\alpha }f\left( t\right) =\frac{1}{\Gamma \left( 1-\alpha \right) }%
\left( \frac{f\left( t\right) }{\left( t-a\right) ^{\alpha }}+\alpha
\int_{a}^{t}\frac{f\left( t\right) -f\left( s\right) }{\left( t-s\right)
^{\alpha +1}}ds\right) 1_{(a,b)}(t)  \label{1.1}
\end{equation*}%
and
\begin{equation*}
D_{b-}^{\alpha }f\left( t\right) =\frac{\left( -1\right) ^{\alpha }}{\Gamma
\left( 1-\alpha \right) }\left( \frac{f\left( t\right) }{\left( b-t\right)
^{\alpha }}+\alpha \int_{t}^{b}\frac{f\left( t\right) -f\left( s\right) }{%
\left( s-t\right) ^{\alpha +1}}ds\right) 1_{(a,b)}(t)   \label{1.2}
\end{equation*}%
for almost all $t\in(a,b)$ (the convergence of the integrals at the singularity $%
s=t$ holds point-wise for almost all $t\in \left( a,b\right) $ if $p=1$ and
moreover in $L^{p}$-sense if $1<p<\infty $).

\bigskip
For any $\lambda \in (0,1)$, we denote by $C^{\lambda }([a,b])$ the space of $%
\lambda $-H\"{o}lder continuous functions on the interval $[a,b]$.

From Theorems 3.5 and 3.6 in \cite{SKM}, we have:

\begin{itemize}
\item[(i)] If $\alpha <\frac{1}{p}$ and $q=\frac{p}{1-\alpha p}$ then
\begin{equation*}
I_{a+}^{\alpha }\left( L^{p}([a,b])\right) =I_{b-}^{\alpha }\left( L^{p}([a,b])\right)
\subset L^{q}\left( [a,b]\right) .
\end{equation*}

\item[(ii)] If $\alpha >\frac{1}{p}$ then%
\begin{equation*}
I_{a+}^{\alpha }\left( L^{p}([a,b])\right) \,\cup \,I_{b-}^{\alpha }\left(
L^{p}([a,b])\right) \subset C^{\alpha -\frac{1}{p}}\left([ a,b]\right) .
\end{equation*}
\item[(iii)] If $\beta>\alpha$, then
\[
C^\beta([a,b])\subset I^\alpha_{a+}(L^p([a,b])),\quad\quad \forall p>1.
\]
\end{itemize}

The following inversion formulas hold:
\begin{eqnarray*}
I_{a+}^{\alpha }\left( D_{a+}^{\alpha }f\right) &=&f,\quad \quad \;\forall
f\in I_{a+}^{\alpha }\left( L^{p}([a,b])\right)  \label{1.4} \\
I_{b-}^{\alpha }\left( D_{b-}^{\alpha }f\right) &=&f,\quad \quad \;\forall
f\in I_{b-}^{\alpha }\left( L^{p}([a,b])\right)  \label{1.3}
\end{eqnarray*}%
and
\begin{equation*}
D_{a+}^{\alpha }\left( I_{a+}^{\alpha }f\right) =f,\quad D_{b-}^{\alpha
}\left( I_{b-}^{\alpha }f\right) =f,\quad \;\forall f\in L^{1}\left(
[a,b]\right) \,.  \label{1.5}
\end{equation*}%

\subsection{Representation of fractional Brownian motion on an interval} $\mbox{}$

In the following sections, let $T>0$ be a fixed number.

 \begin{definition}
A centered Gaussian process $B^H=\{B^H_t;\,t\in[0,T]\}$ is called fractional Brownian motion (fBm for short) with Hurst parameter $H\in(0,1)$ if it has the covariance function
\begin{equation}\label{covariance}
R_H(s, t)=\mathbb{E}(B^{H}_sB^{H}_t)=\frac{1}{2}(s^{2H}+t^{2H}-|t-s|^{2H}),
\end{equation}
for all $s,\,t\in[0,T]$.
\end{definition}

For $H=\frac{1}{2}$, the process $B^{\frac{1}{2}}$ is a standard Brownian motion. For $H\neq\frac{1}{2}$, the fBm $B^H$ is not a semimartingale.
 It follows from (\ref{covariance}) that
\begin{equation*}
\mathbb{E}(\vert B_t^H-B_s^H\vert^2)=\vert t-s\vert^{2H}.
\end{equation*}
Furthermore, by Kolmogorov's continuity criterion, $B^H$ is H\"{o}lder continuous of order $\beta$ for all $\beta<H$.

 Let $F(a,b,c;z)$ denote the Gauss hypergeometric function defined for any $a,\,b,\,c,\, z\in\mathbb{C}$ with $|z|<1$ and $c\neq 0,-1,-2,\dots$ by
\[
F(a,b,c;z)=\sum_{k=0}^{\infty}\frac{(a)_k(b)_k}{(c)_k}z^k,
\]
where $(a)_0=1$ and $(a)_k=a(a+1)\dots(a+k-1)$ is the Pochhammer symbol.

Let
$B^H=\{B_t^H,\,t\in[0,T]\}$ be a fractional Brownian motion (fBm for short)  with Hurst parameter $H\in(0,1)$  on a complete probability space $(\Omega,\mathcal{F},\mathbb{P})$. The following integral representation is given in  \cite{DU}
\begin{equation*}
B_t^H=\int_0^TK_H(t,s)dB_s,
\end{equation*}
where $B=\{B_t,\,t\in [0,T]\}$ is a standard Brownian motion and
 \begin{equation}\label{K-H}
K_H(t,s)=c_H(t-s)^{H-\frac{1}{2}}F\left(H-\frac{1}{2},\frac{1}{2}-H,H+\frac{1}{2};1-\frac{t}{s}\right)\mathbf{1}_{[0,t]}(s),
\end{equation}
 with $c_H=\left[\frac{2H\Gamma\left(\frac{3}{2}-H\right)}{\Gamma(2-2H)\Gamma\left(H+\frac{1}{2}\right)}\right]^{1/2}$.
 
Theorem 5.2 in \cite{nvv} gives an alternative expression for $K_H(t,s)$
 \begin{equation}\label{K-H-1}
 K_H(t,s)=c_H\left[\left(\frac{t}{s}\right)^{H-\frac12}(t-s)^{H-\frac12}-(H-\frac12)s^{\frac12-H}\int_s^tu^{H-\frac32}(u-s)^{H-\frac12}\right]\mathbf{1}_{[0,t]}(s).
 \end{equation}
 
For any $H\in(0,1)$, consider the integral transform
\begin{eqnarray}
(\mathcal{K}_Hf)(t)&=&\int_0^TK_H(t,s)f(s)ds \label{K-H-mapping} \\
&=&c_H\int_0^t (t-s)^{H-\frac{1}{2}}F\left(\frac{1}{2}-H, H-\frac{1}{2}, H+\frac{1}{2};1-\frac{t}{s}\right) f(s)ds. \nonumber
\end{eqnarray}
Then, we have the following important fact (cf. Theorem 2.1 in \cite{DU} and (10.22) in \cite{SKM}).
\begin{lemma}\label{isom}
The operator $\mathcal{K}_H$ is an isomorphism from $L^2([0,T])$ onto $I_{0+}^{H+\frac{1}{2}}(L^2([0,T]))$ and it can be expressed in terms of fractional integrals as follows
\begin{eqnarray}\mathcal{K}_H f&=&c_HI^1_{0+}u^{H-\frac{1}{2}}I^{H-\frac{1}{2}}_{0+}u^{\frac{1}{2}-H}f,\quad \hbox{if}\ H> \frac{1}{2},\label{equ-1-1}\\
\mathcal{K}_Hf&=&c_HI^{2H}_{0+}u^{\frac{1}{2}-H}I^{\frac{1}{2}-H}_{0+}u^{H-\frac{1}{2}}f,\quad \hbox{if}\ H\leq\frac{1}{2}.\label{equ-1-2}
\end{eqnarray}
\end{lemma}

From \eqref{equ-1-1} and \eqref{equ-1-2}, the inverse operator $\mathcal{K}_H^{-1}$ is given by
\begin{align}
\mathcal{K}_H^{-1}h=&\ c_H^{-1}s^{H-\frac{1}{2}}D_{0+}^{H-\frac{1}{2}}s^{\frac{1}{2}-H}h^\prime, \quad \hbox{if}\ H> \frac{1}{2},\label{equ-1-3}\\
\mathcal{K}_H^{-1}h=&\ c_H^{-1}s^{\frac{1}{2}-H}D_{0+}^{\frac{1}{2}-H}s^{H-\frac{1}{2}}D_{0+}^{2H}h, \quad \hbox{if}\ H\leq \frac{1}{2},\label{equ-1-4}
\end{align}
for all $h\in I^{H+\frac{1}{2}}_{0+}(L^2([0,T]))$, where $h^\prime$ is the derivative of $h$ if $h$ is absolutely continuous.
For the case $H\leq\frac{1}{2}$, if $h$ is absolutely continuous, we can apply (10.6) in \cite{SKM} to get
\begin{equation}\label{equ-1-5}
\mathcal{K}_H^{-1}h=c_H^{-1}s^{H-\frac{1}{2}}I_{0+}^{\frac{1}{2}-H}s^{\frac{1}{2}-H}h^\prime.
\end{equation}

The following lemma, combined with 
the statements in the latter half of Lemma 
\ref{isom}, ensures that $ \th_2 $ in \eqref{h-2-1} is well-defined. 
\begin{lemma}\label{inverse}
We have
\begin{equation}
\int_0^\cdot h_2(s, X_s,Y_s)ds\in I^{H+\frac{1}{2}}_{0+}(L^2([0,T])) \quad \mbox{ almost surely.}
\end{equation}
\end{lemma}
\begin{proof} The case of $H=\frac12$ is trivial.

For the case of $H<\frac{1}{2}$, by \eqref{comp-1} we have
\[
\int_0^\cdot h_2(s, X_s,Y_s)ds=I^{1}_{0^+}h_2(\cdot,X_\cdot,Y_\cdot)=I^{H+\frac12}_{0^+}I^{\frac12-H}_{0^+}h_2(\cdot,X_\cdot,Y_\cdot).
\]
Note that \eqref{e-2-4} in the proof of Theorem \ref{holder} and the linear growth condition \eqref{linear} on $h_2$ imply that $h_2(\cdot, X_\cdot,Y_\cdot)$ is in $L^2([0,T])$.
Using  (i) in Section \ref{fid} with $\alpha=\frac12-H$, $p=2$ and $q=\frac{p}{1-\alpha p}=\frac1H>2$, we can show that $I^{\frac12-H}_{0^+}h_2(\cdot,X_\cdot,Y_\cdot)\in L^q([0,T])\subset  L^2([0,T])$ which implies the result.

For the case of $H>\frac{1}{2}$, by \eqref{comp-1}, we need $h_2(\cdot, X_\cdot,Y_\cdot)\in I^{H-\frac{1}{2}}_{0+}(L^2([0,T]))$, which is implied by (iii) with $\alpha=H-\frac12<\gamma<\frac12$ in Section \ref{fid} and the fact that $h_2(\cdot, X_\cdot,Y_\cdot)\in C^{\gamma}([0,T])$ from the result in Theorem \ref{holder}.
\end{proof}
Since $\int_0^\cdot h_2(s, X_s,Y_s)ds\in I^{H+\frac{1}{2}}_{0+}(L^2([0,T]))$ almost surely and the operator $\mathcal{K}_H^{-1}$ preserves adaptability, there exist adapted stochastic processes $\tilde{h}_1,\, \tilde{h}_2\in L^2([0,T])$ such that
\begin{equation}\label{h-2}
\tilde{h}_2(t)=\mathcal{K}_H^{-1}\left(\int_0^\cdot h_2(s, X_s,Y_s)ds\right)(t),
\end{equation}
and
\begin{equation}\label{h-1-2-a}
\rho \,\tilde{h}_2(t)+\sqrt{1-\rho^2}\,\tilde{h}_1(t)=h_1(t, X_t,Y_t).
\end{equation}

Similarly, we have
\begin{lemma}\label{wdn2}
The two functions $\hat h_1$ and $\hat h_2$ determined by the equations \eqref{hat-2} and \eqref{hat-1} are well-defined.
\end{lemma}

\begin{proof}
Note that by applying the Cauchy-Schwarz
inequality we have for any $s,t\in[0,T]$
\begin{eqnarray*}
\left|\int_0^tK_H(T,u)du-\int_0^sK_H(T,u)du\right|&=&|\mathbb{E}[(B_t-B_s)B_T^H]|\\
&\leq&\left(\mathbb{E}[|B_t-B_s|^2]\right)^{\frac{1}{2}}\left(\mathbb{E}[|B^H_T|^2]\right)^{\frac{1}{2}}\\
&\leq& T^H|t-s|^{\frac12},
\end{eqnarray*}
and 

\begin{eqnarray*}
&& |R_H(t,T)-R_H(s,T)| = |\mathbb{E}[(B_t^H-B_s^H)B_T^H]| \\
&\leq& \left(\mathbb{E}[|B_t^H-B_s^H|^2]\right)^{\frac{1}{2}}\left(\mathbb{E}[|B^H_T|^2]\right)^{\frac{1}{2}}
= T^H|t-s|^H.
\end{eqnarray*}

It follows from \eqref{eqn:EXtxy} and \eqref{eqn:EYtxy} that the expectations $\tEof{X_t^{x,y}}$ and $\tEof{Y_t^{x,y}}$
are $\alpha$-H\"{o}lder continuous in $t$ of any order $\alpha<\min\{H,\frac12\}$. Hence, as in the proof of Lemma \ref{inverse}, we can show that $\int_0^\cdot h_2(s, \tE{X_s^{x,y}},\tE{Y_s^{x,y}})ds\in I^{H+\frac{1}{2}}_{0^+}(L^2([0,T]))$. We conclude that the two deterministic functions $\hat{h}_1$ and $\hat{h}_2$ are well-defined.
\end{proof}

\subsection{The conditional expectation of Gaussian random vectors and Gaussian bridges} $\mbox{}$

Let $\mathbf{X}=(X_1,\cdots,X_n)^\prime$ and $\mathbf{Y}=(Y_1,\cdots,Y_m)^\prime$ be joint Gaussian random vectors and $\mathbf{Z}=(X_1,\cdots,X_n,Y_1,\cdots,Y_m)^\prime$. Denote the expectations of $\mathbf{X}$, $\mathbf{Y}$ and the covariance matrix for $\mathbf{Z}$ by
\[
\mathbb{E}[\mathbf{X}]=\mu_\mathbf{X},\ \mathbb{E}[\mathbf{Y}]=\mu_\mathbf{Y}
\]
and
\[
\mathbf{\Sigma}=Cov\left[\left(\begin{array}{c}\mathbf{X}\\ \mathbf{Y}\end{array}\right)\right]=\left(\begin{array}{cc}\mathbf{\Sigma_{XX}} &\mathbf{\Sigma_{XY}}\\ \mathbf{\Sigma_{YX}}&\mathbf{\Sigma_{YY}}\end{array}\right)
\]

The following lemma gives the conditional distribution of Gaussian random vectors.
\begin{lemma}\label{conditional}
Suppose that the covariance matrix $\mathbf{\Sigma}$ is positive definite. Then, the conditional distribution of $\mathbf{X}$ given that $\mathbf{Y}=\mathbf{y}$ is $n$-dimensional Gaussian with expectation
\[
\mathbb{E}[\mathbf{X}|\mathbf{Y}=\mathbf{y}]=\mu_\mathbf{X}+\mathbf{\Sigma_{XY}}\mathbf{\Sigma_{YY}}^{-1}(\mathbf{y}-\mu_\mathbf{Y})
\]
and covariance matrix
\[
Cov[\mathbf{X}|\mathbf{Y}=\mathbf{y}]=\mathbf{\Sigma_{XX}}-\mathbf{\Sigma_{XY}}\mathbf{\Sigma_{YY}}^{-1}\mathbf{\Sigma_{YX}}.
\]
Moreover, the Gaussian vector $\mathbf{X}$ has the following decomposition
\[
\bX = \mu_\mathbf{X}+\mathbf{\Sigma_{XY}}\mathbf{\Sigma_{YY}}^{-1}(\mathbf{Y}-\mu_\mathbf{Y})+\mathbf{V},
\]
where the random vector $\mathbf{V}$ is $n$-dimensional Gaussian with zero expectation and the following covariance matrix
\[
Cov[\mathbf{V}]=\mathbf{\Sigma_{XX}}-\mathbf{\Sigma_{XY}}\mathbf{\Sigma_{YY}}^{-1}\mathbf{\Sigma_{YX}}.
\]
\subsection{Some estimates on $\tilde{h}_i$ and $\hat{h}_i$, $i=1, 2$} $\mbox{}$

We will give some important estimates on $\tilde{h}_i$ and $\hat{h}_i$, $i=1, 2$, in the cases $H>\frac12$ and $H\leq \frac12$ respectively.
\begin{lemma}\label{lem-7}
\begin{enumerate}
\item
In the case of $H>\frac12$, there exists a constant $C$ depending on $x_0$, $y_0$, $\rho$ and the constants $L$ in \eqref{Hold} and $K$ in \eqref{linear} such that, for any $0<\epsilon<1-H$,
\begin{itemize}
\item[(a)] 
\begin{eqnarray}
&&\int_0^T(|\tilde{h}_1(t)|^2+|\tilde{h}_2(t)|^2)dt\nonumber\\
&\leq& C(1+\Vert \tilde{B}\Vert^2_\infty+\Vert \tilde{W}\Vert^2_\infty+\Vert \tilde{B}^H\Vert^2_\infty+\Vert\tilde{B}\Vert^2_{\frac12-\epsilon}+\Vert\tilde{W}\Vert^2_{\frac12-\epsilon}+\Vert\tilde{B}^H\Vert^2_{\frac12-\epsilon})T^{2-2H};\nonumber\\
\end{eqnarray}

\item[(b)] 
\begin{equation}\label{Jun-2-43}
\int_0^T(|\hat{h}_1(t)|^2+|\hat{h}_2(t)|^2)dt\leq 
C(1+|x-x_0|^2+|y-y_0|^2)T^{2-2H}+C|y-y_0|^2T^{3-4H};
\end{equation}
\item[(c)]
 \begin{eqnarray}\label{Jun-2-44}
&&\int_0^T(|\tilde{h}_1(t)-\hat{h}_1(t)|^2+|\tilde{h}_2(t)-\hat{h}_2(t)|^2)dt\nonumber\\
&\leq&C(1+\Vert\tilde{B}\Vert_{\frac12-\epsilon}^2+\Vert\tilde{W}\Vert^2_{\frac12-\epsilon}+\Vert\tilde{B}^H\Vert^2_{\frac12-\epsilon}+\Vert \tilde{B}\Vert^2_{\infty}+\Vert \tilde{W}\Vert^2_{\infty}+\Vert \tilde{B}^H\Vert^2_{\infty})T^{2-2H}\nonumber\\
&&+C(|x-x_0|^2+|y-y_0|^2)T^{2-2H}+C|y-y_0|^2T^{3-4H}.
\end{eqnarray}

\end{itemize}
\item In the case of $H\leq\frac12$, there exists a constant $C$ depending on $x_0$, $y_0$, $\rho$ and the constants $L$ in \eqref{Hold} and $K$ in \eqref{linear} such that, for any $0<\epsilon<1-H$,

\begin{itemize}
\item[(e)]
\bea
\int_0^T\vert(\tilde{h}_1(s)\vert^2+\vert\tilde{h}_2(s)\vert^2)ds&\leq& C\left(1+\Vert \tilde{B}\Vert_\infty^2+\Vert \tilde{W}\Vert_\infty^2+\Vert \tilde{B}^H\Vert_\infty^2\right)T\nonumber\\
&\leq&C\left(1+\Vert \tilde{B}\Vert_\infty^2+\Vert \tilde{W}\Vert_\infty^2+\Vert \tilde{B}^H\Vert_\infty^2\right)T^{2H}.
\eea
\item[(f)]
\begin{equation}
\int_0^T(\vert \hat{h}_1(t)\vert^2+\vert \hat{h}_2(t)\vert^2)dt\leq C(1+|x-x_0|^2+|y-y_0|^2)T^{2H}.
\end{equation}
\item[(g)]
\bea
&&\int_0^T(|\tilde{h}_1(t)-\hat{h}_1(t)|^2+|\tilde{h}_2(t)-\hat{h}_2(t)|^2)dt\nonumber\\
&\leq&C(\Vert \tilde{B}\Vert_{\infty}^2+\Vert \tilde{W}\Vert_{\infty}^2+\Vert \tilde{B}^H\Vert_{\infty}^2+|x-x_0|^2+|y-y_0|^2)T^{2H}.
\eea

\end{itemize}
\end{enumerate}
\end{lemma}
\begin{proof} \underline{Case of $H>\frac12$:} We choose an arbitrary small $0<\epsilon<1-H$ (note that $1-H<\frac12$).

In the following, we will use $C$ to denote a generic constant which is dependent on $x_0$, $y_0$, $\rho$ and the constants $L$ in \eqref{Hold} and $K$ in \eqref{linear} but  independent of $T$ and $(x,y)$.

For any $s,\, t\in[0,T]$, by \eqref{til-X}, \eqref{til-Y}, the linear growth condition \eqref{linear}, the H\"{o}lder continuity condition \eqref{Hold} and the Lipschitz condition \eqref{Lip} on $h_i$, $i=1,2$, we have
\begin{eqnarray}\label{ap-2-44}
|h_i(t,X_t,Y_t)|&\leq & C(1+\Vert \tilde{B}\Vert_\infty+\Vert \tilde{W}\Vert_\infty+\Vert \tilde{B}^H\Vert_\infty),
\end{eqnarray}
and
\begin{eqnarray}\label{ap-2-45}
&&|h_2(t,X_t,Y_t)-h_2(s,X_s,Y_s)|\nonumber\\
&\leq& C|t-s|^\gamma+C(\Vert\tilde{B}\Vert_{\frac12-\epsilon}+\Vert\tilde{W}\Vert_{\frac12-\epsilon}+\Vert\tilde{B}^H\Vert_{\frac12-\epsilon})|t-s|^{\frac12-\epsilon}.
\end{eqnarray}
Considering \eqref{Jan-2-14}, we get by \eqref{ap-2-44} and \eqref{ap-2-45}\begin{eqnarray}\label{ap-2-46}
|a(t)|&\leq& C(1+\Vert \tilde{B}\Vert_\infty+\Vert \tilde{W}\Vert_\infty+\Vert \tilde{B}^H\Vert_\infty)t^{\frac12-H}+C\int_0^t(t-s)^{\gamma-H-\frac12}ds\nonumber\\
&&+C(\Vert\tilde{B}\Vert_{\frac12-\epsilon}+\Vert\tilde{W}\Vert_{\frac12-\epsilon}+\Vert\tilde{B}^H\Vert_{\frac12-\epsilon})\int_0^t(t-s)^{-H-\epsilon}ds\nonumber\\
&\leq&C(1+\Vert \tilde{B}\Vert_\infty+\Vert \tilde{W}\Vert_\infty+\Vert \tilde{B}^H\Vert_\infty)t^{\frac12-H}\nonumber\\
&&+C(1+\Vert\tilde{B}\Vert_{\frac12-\epsilon}+\Vert\tilde{W}\Vert_{\frac12-\epsilon}+\Vert\tilde{B}^H\Vert_{\frac12-\epsilon}),
\end{eqnarray}
and by \eqref{ap-2-44} and changing of variables we have
\begin{eqnarray}\label{ap-2-47}
|b(t)|&\leq&C(1+\Vert \tilde{B}\Vert_\infty+\Vert \tilde{W}\Vert_\infty+\Vert \tilde{B}^H\Vert_\infty)t^{\frac12-H}\int_0^1\frac{u^{\frac12-H}-1}{(1-u)^{H+\frac12}}du\nonumber\\
&=&\alpha_HC(1+\Vert \tilde{B}\Vert_\infty+\Vert \tilde{W}\Vert_\infty+\Vert \tilde{B}^H\Vert_\infty)t^{\frac12-H},
\end{eqnarray}
where $\alpha_H = \int_0^1\frac{u^{\frac12-H}-1}{(1-u)^{H+\frac12}}du$ is a finite constant. 

Then, \eqref{Jan-2-14}, \eqref{ap-2-46} and \eqref{ap-2-47} imply 
\begin{eqnarray}\label{ap-2-48}
|\tilde{h}_2(t)|&\leq& C(1+\Vert \tilde{B}\Vert_\infty+\Vert \tilde{W}\Vert_\infty+\Vert \tilde{B}^H\Vert_\infty)t^{\frac12-H}\nonumber\\
&&+C(1+\Vert\tilde{B}\Vert_{\frac12-\epsilon}+\Vert\tilde{W}\Vert_{\frac12-\epsilon}+\Vert\tilde{B}^H\Vert_{\frac12-\epsilon}).
\end{eqnarray}
From \eqref{h-1-2-a}, \eqref{ap-2-44} and \eqref{ap-2-48}, we obtain
\begin{eqnarray}\label{ap-2-49}
|\tilde{h}_1(t)|&\leq&C(1+\Vert \tilde{B}\Vert_\infty+\Vert \tilde{W}\Vert_\infty+\Vert \tilde{B}^H\Vert_\infty+\Vert\tilde{B}\Vert_{\frac12-\epsilon}+\Vert\tilde{W}\Vert_{\frac12-\epsilon}+\Vert\tilde{B}^H\Vert_{\frac12-\epsilon})\nonumber\\
&&+ C(1+\Vert \tilde{B}\Vert_\infty+\Vert \tilde{W}\Vert_\infty+\Vert \tilde{B}^H\Vert_\infty)t^{\frac12-H}.
\end{eqnarray}
Thus, by \eqref{ap-2-48} and \eqref{ap-2-49}, one can obtain
\begin{eqnarray}\label{ap-2-52-1}
&&\int_0^T(|\tilde{h}_1(t)|^2+|\tilde{h}_2(t)|^2)dt\nonumber\\
&\leq&C(1+\Vert \tilde{B}\Vert^2_\infty+\Vert \tilde{W}\Vert^2_\infty+\Vert \tilde{B}^H\Vert^2_\infty+\Vert\tilde{B}\Vert^2_{\frac12-\epsilon}+\Vert\tilde{W}\Vert^2_{\frac12-\epsilon}+\Vert\tilde{B}^H\Vert^2_{\frac12-\epsilon})T\nonumber\\
&&+C(1+\Vert \tilde{B}\Vert^2_\infty+\Vert \tilde{W}\Vert^2_\infty+\Vert \tilde{B}^H\Vert^2_\infty)T^{2-2H}\nonumber\\
&\leq&C(1+\Vert \tilde{B}\Vert^2_\infty+\Vert \tilde{W}\Vert^2_\infty+\Vert \tilde{B}^H\Vert^2_\infty+\Vert\tilde{B}\Vert^2_{\frac12-\epsilon}+\Vert\tilde{W}\Vert^2_{\frac12-\epsilon}+\Vert\tilde{B}^H\Vert^2_{\frac12-\epsilon})T^{2-2H}.\nonumber\\
\end{eqnarray}

Now let us prove Part $(b)$. From \eqref{m11} - \eqref{m22},  it implies the following estimates
\begin{eqnarray}
&&|m_{11}(t;T)|\leq \frac{1+\rho\rho_H}{\bar{\rho}_H^2},\label{est-m11}\\
&&|m_{12}(t;T)|\leq \frac{\rho+\rho_H}{\bar{\rho}_H^2\ T^{H-\frac12}},\\
&&|m_{21}(t;T)|\leq \frac{2\rho_H\ T^{H-\frac12}}{\bar{\rho}_H^2},\\
&&|m_{22}(t;T)|\leq \frac{1+\rho_H^2}{\bar{\rho}_H^2}.\label{est-m22}
\end{eqnarray}
Note that in the case of $H>\frac12$, we further have 
\begin{equation}
|m_{21}(t;T)|\leq \frac{2\rho_H}{\bar{\rho}_H^2}.\label{est-m12-1}
\end{equation} 
Moreover, for all $s, t\in[0,T]$, there exists a constant $C$ depending on $\rho$ and $H$ such that
\begin{eqnarray}\label{est-m11st}
|m_{11}(t;T)-m_{11}(s;T)|&\leq& C \left(\frac{|t-s|}{T}+\frac{|\int_s^tK_H(T,u)du|}{T^{H+\frac12}}\right)\nonumber\\
&\leq& C \left(\frac{|t-s|}{T}+\frac{|\tE{B_T^H(B_t-B_s)}|}{T^{H+\frac12}}\right)\nonumber\\
&\leq&C \left(\frac{|t-s|}{T}+\frac{|t-s|^{\frac12}}{T^{\frac12}}\right)\leq C\frac{|t-s|^{\frac12}}{T^{\frac12}},
\end{eqnarray}
\bea
|m_{12}(t;T)-m_{12}(s;T)|&\leq& C \left(\frac{|t-s|}{T^{H+\frac12}}+\frac{|\int_s^tK_H(T,u)du|}{T^{2H}}\right)\nonumber\\
&\leq&C \left(\frac{|t-s|}{T^{H+\frac12}}+\frac{|t-s|^{\frac12}}{T^H}\right)\leq C\frac{|t-s|^{\frac12}}{T^{H}},
\eea
\bea
|m_{21}(t;T)-m_{21}(s;T)|&\leq& C \left(\frac{|t^{H+\frac12}-s^{H+\frac12}|}{T}+\frac{|R_H(t,T)-R_H(s,T)|}{T^{H+\frac12}}\right)\nonumber\\
&\leq& C\left(\frac{|t-s|}{T^{\frac32-H}}+\frac{|\tE{(B_T^H(B_t^H-B_s^H)}|}{t^{H+\frac12}}\right)\nonumber\\
&\leq&C\left(\frac{|t-s|}{T^{\frac32-H}}+\frac{|t-s|^{H}}{T^{\frac12}}\right)\leq C\frac{|t-s|^{H}}{T^{\frac12}}\leq C\frac{|t-s|^{\frac12}}{T^{\frac12}},\nonumber\\
\eea
and 
\bea\label{est-m22st}
|m_{22}(t;T)-m_{22}(s;T)|&\leq& C \left(\frac{|t^{H+\frac12}-s^{H+\frac12}|}{T^{H+\frac12}}+\frac{|R_H(t,T)-R_H(s,T)|}{T^{2H}}\right)\nonumber\\
&\leq& C\left(\frac{|t-s|}{T}+\frac{|t-s|^H}{T^H}\right)\leq C\frac{|t-s|^H}{T^H}.
\eea

Similar to \eqref{Jan-2-14}, we can write
\begin{equation}\label{hat-h-ab}
\hat{h}_2(t)=\frac{c_H^{-1}}{\Gamma(\frac32-H)}(\hat{a}(t)+\hat{b}(t)),
\end{equation}
where
\[
\hat{a}(t)=t^{\frac12-H}\bar{h}_2(t)+(H-\frac12)\int_0^t\frac{\bar{h}_2(t)-  \bar{h}_2(s)}{(t-s)^{H+\frac12}}ds,
\]
and
\[
\hat{b}(t)=(H-\frac12)t^{H-\frac12}\int_0^t\frac{(t^{\frac12-H}-s^{\frac12-H})\bar{h}_2(s)}{(t-s)^{H+\frac12}}ds.
\]
By the linear growth condition \eqref{linear}, \eqref{eqn:EXtxy} - \eqref{m22} and \eqref{est-m11} - \eqref{est-m12-1}, we can see
\begin{equation}\label{ap-2-51}
|\bar{h}_i(t)|\leq C(1+|x-x_0|+|y-y_0|+|y-y_0|T^{\frac12-H}),\ i=1,2.
\end{equation}

From the H\"{o}lder continuity condition \eqref{Hold}, the Lipschitz condition \eqref{Lip} on $h_2$, the definition of $\bar{h}_2$, \eqref{eqn:EXtxy} - \eqref{eqn:EYtxy} and \eqref{est-m11st} - \eqref{est-m22st}, it is easy to show
\begin{equation}\label{ap-2-56-1}
|\bar{h}_2(t)-\bar{h}_2(s)|\leq C|t-s|^\gamma+C|x-x_0|\frac{|t-s|^{\frac12}}{T^{\frac12}}+C|y-y_0|\left(\frac{|t-s|^{\frac12}+|t-s|^H}{T^H}\right).
\end{equation}
Thus, analogue to the proofs of \eqref{ap-2-46} and \eqref{ap-2-47}, we obtain
\bea\label{bar-a}
|\hat{a}(t)|&\leq&C(1+|x-x_0|+|y-y_0|+|y-y_0|T^{\frac12-H})t^{\frac12-H}+Ct^{\gamma-H+\frac12}\nonumber\\
&&+C|x-x_0|\frac{t^{1-H}}{T^{\frac12}}+C|y-y_0|\left(\frac{t^{1-H}+t^{\frac12}}{T^H}\right),
\eea
and
\bea\label{bar-b}
|\hat{b}(t)|&\leq&C(1+|x-x_0|+|y-y_0|+|y-y_0|T^{\frac12-H})t^{\frac12-H}.
\eea
So, from \eqref{hat-h-ab}, \eqref{bar-a} and \eqref{bar-b} it implies
\bea\label{Jun-hat-2}
|\hat{h}_2(t)|&\leq&C(1+|x-x_0|+|y-y_0|+|y-y_0|T^{\frac12-H})t^{\frac12-H}+Ct^{\gamma-H+\frac12}\nonumber\\
&&+C|x-x_0|\frac{t^{1-H}}{T^{\frac12}}+C|y-y_0|\left(\frac{t^{1-H}+t^{\frac12}}{T^H}\right).
\eea
Then, by \eqref{hat-h-ab}, \eqref{ap-2-51} and \eqref{Jun-hat-2} one has
\bea\label{Jun-hat-1}
|\hat{h}_1(t)|&\leq&C(1+|x-x_0|+|y-y_0|+|y-y_0|T^{\frac12-H})(1+t^{\frac12-H})+Ct^{\gamma-H+\frac12}\nonumber\\
&&+C|x-x_0|\frac{t^{1-H}}{T^{\frac12}}+C|y-y_0|\left(\frac{t^{1-H}+t^{\frac12}}{T^H}\right).
\eea

Therefore, taking into account of \eqref{Jun-hat-2} and \eqref{Jun-hat-1}, we can prove
\beaa
&&\int_0^T(|\hat{h}_1(t)|^2+|\hat{h}_2(t)|^2)dt\nonumber\\
&\leq &C(1+|x-x_0|^2+|y-y_0|^2+|y-y_0|^2T^{1-2H})(T+T^{2-2H})+CT^{2\gamma-2H+2}\nonumber\\
&&+C|x-x_0|^2T^{2-2H}+C|y-y_0|^2(T^{3-4H}+T^{2-2H})\nonumber\\
&\leq& C(1+|x-x_0|^2+|y-y_0|^2)T^{2-2H}+C|y-y_0|^2T^{3-4H},
\eeaa
since $\gamma>0$.

Now let us prove Part (c). By the definitions of $\tilde{h}_2$ and $\hat{h}_2$ in \eqref{h-2} and \eqref{hat-2}, and \eqref{equ-1-3} we can write
\begin{eqnarray}\label{ap-2-61}
\tilde{h}_2(t)-\hat{h}_2(t)&=&\frac{c_H^{-1}}{\Gamma(\frac32-H)}(\bar{a}(t)+\bar{b}(t)),
\end{eqnarray}
where
\begin{eqnarray*}
\bar{a}(t)&=&\left(H-\frac12\right)\int_0^t\frac{(h_2(t,X_t,Y_t)-\bar{h}_2(t))-(h_2(s,X_s,Y_s)-\bar{h}_2(s))}{(t-s)^{H+\frac12}}ds\nonumber\\
&&+
t^{\frac12-H}(h_2(t,X_t,Y_t)-\bar{h}_2(t)),
\end{eqnarray*}
and
\[
\bar{b}(t)=\left(H-\frac12\right)t^{H-\frac12}\int_0^t\frac{(t^{\frac12-H}-s^{\frac12-H})(h_2(s,X_s,Y_s)-\bar{h}_2(s))}{(t-s)^{H+\frac12}}ds.
\]
For any $t\in[0,T]$, from the Lipschitz condition on $h_i$, $i=1,2$, \eqref{til-X}, \eqref{til-Y}, \eqref{eqn:EXtxy} and \eqref{est-m11} - \eqref{est-m12-1}, it implies
\begin{eqnarray}\label{ap-2-59}
&&|h_i(t,X_t,Y_t)-\bar{h}_i(t)|\nonumber\\
&\leq&C\left(|X_t-x_0|+|Y_t-y_0|+(|m_{11}(t;T)|+|m_{21}(t;T)|)|x-x_0|\right.\nonumber\\
&&\left.+(|m_{12}(t;T)|+|m_{22}(t;T)|)|y-y_0|\right)\nonumber\\
&\leq & C(\Vert \tilde{B}\Vert_{\infty}+\Vert \tilde{W}\Vert_{\infty}+\Vert \tilde{B}^H\Vert_{\infty})+C(|x-x_0|+|y-y_0|+|y-y_0|T^{\frac12-H}).\nonumber\\
\end{eqnarray}
For any $s,t\in[0,T]$, the inequalities \eqref{ap-2-59}, \eqref{ap-2-45}, \eqref{ap-2-56-1} and the calculations in \eqref{ap-2-46} and \eqref{bar-a} yield
\begin{eqnarray}\label{ap-2-62}
|\bar{a}(t)|&\leq &C(1+\Vert\tilde{B}\Vert_{\frac12-\epsilon}+\Vert\tilde{W}\Vert_{\frac12-\epsilon}+\Vert\tilde{B}^H\Vert_{\frac12-\epsilon})\nonumber\\
&&+C|x-x_0|\frac{t^{1-H}}{T^{\frac12}}+C|y-y_0|\left(\frac{t^{1-H}+t^{\frac12}}{T^H}\right)\nonumber\\
&&+C(\Vert \tilde{B}\Vert_{\infty}+\Vert \tilde{W}\Vert_{\infty}+\Vert \tilde{B}^H\Vert_{\infty}+|x-x_0|+|y-y_0|+|y-y_0|T^{\frac12-H})t^{\frac12-H}.\nonumber\\
\end{eqnarray}
By \eqref{ap-2-59} and a change of variables we can prove that
\begin{eqnarray}\label{ap-2-63}
|\bar{b}(t)|&\leq&  C(\Vert \tilde{B}\Vert_{\infty}+\Vert \tilde{W}\Vert_{\infty}+\Vert \tilde{B}^H\Vert_{\infty}+|x-x_0|+|y-y_0|+|y-y_0|T^{\frac12-H})t^{\frac12-H}.\nonumber\\
\end{eqnarray}

Thus, by \eqref{ap-2-61}-\eqref{ap-2-63},  we have 
\bea\label{ap-2-64}
&&\int_0^T|\tilde{h}_2(t)-\hat{h}_2(t)|^2dt\nonumber\\
&\leq&C(1+\Vert\tilde{B}\Vert_{\frac12-\epsilon}^2+\Vert\tilde{W}\Vert^2_{\frac12-\epsilon}+\Vert\tilde{B}^H\Vert^2_{\frac12-\epsilon})T+C|y-y_0|^2T^{3-4H}\nonumber\\
&&+C(\Vert \tilde{B}\Vert^2_{\infty}+\Vert \tilde{W}\Vert^2_{\infty}+\Vert \tilde{B}^H\Vert^2_{\infty}+|x-x_0|^2+|y-y_0|^2)T^{2-2H}\nonumber\\
&\leq&C(1+\Vert\tilde{B}\Vert_{\frac12-\epsilon}^2+\Vert\tilde{W}\Vert^2_{\frac12-\epsilon}+\Vert\tilde{B}^H\Vert^2_{\frac12-\epsilon}+\Vert \tilde{B}\Vert^2_{\infty}+\Vert \tilde{W}\Vert^2_{\infty}+\Vert \tilde{B}^H\Vert^2_{\infty})T^{2-2H}\nonumber\\
&&+C(|x-x_0|^2+|y-y_0|^2)T^{2-2H}+C|y-y_0|^2T^{3-4H},
\eea
which implies together from \eqref{h-1-2-a}, \eqref{hat-1} and \eqref{ap-2-59}
\begin{eqnarray}\label{ap-2-65}
&&\int_0^T|\tilde{h}_1(t)-\hat{h}_1(t)|^2dt\nonumber\\
&\leq& C(1+\Vert\tilde{B}\Vert_{\frac12-\epsilon}^2+\Vert\tilde{W}\Vert^2_{\frac12-\epsilon}+\Vert\tilde{B}^H\Vert^2_{\frac12-\epsilon}+\Vert \tilde{B}\Vert^2_{\infty}+\Vert \tilde{W}\Vert^2_{\infty}+\Vert \tilde{B}^H\Vert^2_{\infty})T^{2-2H}\nonumber\\
&&+C(|x-x_0|^2+|y-y_0|^2)T^{2-2H}+C|y-y_0|^2T^{3-4H}.
\end{eqnarray}
Hence, by \eqref{ap-2-64} and \eqref{ap-2-65} we obtain \eqref{Jun-2-44}.

\underline{Case of $H\leq\frac12$:} Proof of Part (e): From \eqref{equ-1-5}, \eqref{linear}, \eqref{til-X}, \eqref{til-Y} and changing of variables it implies 
\bea\label{Jun-2-75}
\vert \tilde{h}_2(t)\vert&=&c_H^{-1} t^{H-\frac{1}{2}}\left\vert\int_0^t(t-s)^{-\frac{1}{2}-H}s^{\frac{1}{2}-H}h_2(s, X_s,Y_s)ds\right\vert\nonumber\\
&\leq&C(1+\Vert \tilde{B}\Vert_{\infty}+\Vert \tilde{W}\Vert_{\infty}+\Vert \tilde{B}^H\Vert_{\infty})B\left(\frac12-H,\frac32-H\right)t^{\frac12-H}\nonumber\\
&=&C(1+\Vert \tilde{B}\Vert_{\infty}+\Vert \tilde{W}\Vert_{\infty}+\Vert \tilde{B}^H\Vert_{\infty}).
\eea
From \eqref{h-1-2-a}, \eqref{Jun-2-75}, the linear growth condition \eqref{linear} on $h_1$,  \eqref{til-X} and \eqref{til-Y}, we obtain
\begin{equation}\label{Jun-2-76}
\vert \tilde{h}_1(t)\vert\leq C \left(1+\Vert B\Vert_\infty+\Vert W\Vert_\infty+\Vert B^H\Vert_\infty\right).
\end{equation}
Thus, \eqref{Jun-2-75} and \eqref{Jun-2-76} imply
\begin{equation*}
\int_0^T\vert(\tilde{h}_1(s)\vert^2+\vert\tilde{h}_2(s)\vert^2)ds\leq C\left(1+\Vert B\Vert_\infty^2+\Vert W\Vert_\infty^2+\Vert B^H\Vert_\infty^2\right)T.
\end{equation*}
Proof of Part (f): Similar to the proof of \eqref{Jun-2-75}, by  \eqref{equ-1-5}, the linear growth condition \eqref{linear} on $h_2$ and \eqref{est-m11} - \eqref{est-m22}, we can show 
\begin{equation}\label{Jun-2-77}
|\bar{h}_2(t)|\leq C(1+|x-x_0|+|y-y_0|+|x-x_0|T^{H-\frac12})
\end{equation}
and hence,
\bea\label{Jun-2-78}
\vert \hat{h}_2(t)\vert&=&c_H^{-1} t^{H-\frac{1}{2}}\left\vert\int_0^t(t-s)^{-\frac{1}{2}-H}s^{\frac{1}{2}-H}\bar{h}_2(s)ds\right\vert\nonumber\\
&\leq& C(1+|x-x_0|+|y-y_0|+|x-x_0|T^{H-\frac12}).
\eea
From \eqref{hat-1}, \eqref{Jun-2-78}, the linear growth condition \eqref{linear} on $h_1$, we have
\begin{equation}\label{Jun-2-79}
\vert \hat{h}_1(t)\vert\leq C(1+|x-x_0|+|y-y_0|+|x-x_0|T^{H-\frac12}).
\end{equation}
Thus, we can obtain
\beaa
\int_0^T(\vert \hat{h}_1(t)\vert^2+\vert \hat{h}_2(t)\vert^2)dt&\leq&C(1+|x-x_0|^2+|y-y_0|^2)T+C|x-x_0|^2T^{2H}\nonumber\\
&\leq&C(1+|x-x_0|^2+|y-y_0|^2)T^{2H}.
\eeaa
Proof of Part (g): From the Lipschitz condition on $h_i$, $i=1,2$, \eqref{til-X}, \eqref{til-Y}, \eqref{eqn:EXtxy} and \eqref{est-m11} - \eqref{est-m22}, it implies
\begin{eqnarray}\label{Jun-2-81}
&&|h_i(t,X_t,Y_t)-\bar{h}_i(t)|\nonumber\\
&\leq&C\left(|X_t-x_0|+|Y_t-y_0|+(|m_{11}(t;T)|+|m_{21}(t;T)|)|x-x_0|\right.\nonumber\\
&&\left.+(|m_{12}(t;T)|+|m_{22}(t;T)|)|y-y_0|\right)\nonumber\\
&\leq & C(\Vert \tilde{B}\Vert_{\infty}+\Vert \tilde{W}\Vert_{\infty}+\Vert \tilde{B}^H\Vert_{\infty})+C(|x-x_0|+|y-y_0|+|x-x_0|T^{H-\frac12}).\nonumber\\
\end{eqnarray}

From  \eqref{Jun-2-81}, the definitions of $\tilde{h}_2$ and $\hat{h}_2$ and \eqref{equ-1-5}, one can easily see
\bea\label{Jun-2-82}
|\tilde{h}_2(t)-\hat{h}_2(t)|&=&c_H^{-1} t^{H-\frac{1}{2}}\left\vert\int_0^t(t-s)^{-\frac{1}{2}-H}s^{\frac{1}{2}-H}(h_2(s, X_s,Y_s)-\bar{h}_2(s))ds\right\vert\nonumber\\
&\leq&C(\Vert \tilde{B}\Vert_{\infty}+\Vert \tilde{W}\Vert_{\infty}+\Vert \tilde{B}^H\Vert_{\infty})+C(|x-x_0|+|y-y_0|+|x-x_0|T^{H-\frac12}).\nonumber\\
\eea
From \eqref{h-1-2-a}, \eqref{hat-1}, \eqref{Jun-2-81} and \eqref{Jun-2-82}, one can also easily get
\bea
|\tilde{h}_1(t)-\hat{h}_1(t)|\leq C(\Vert \tilde{B}\Vert_{\infty}+\Vert \tilde{W}\Vert_{\infty}+\Vert \tilde{B}^H\Vert_{\infty})+C(|x-x_0|+|y-y_0|+|x-x_0|T^{H-\frac12}).\nonumber\\
\eea
Therefore, one can obtain
\beaa
&&\int_0^T(|\tilde{h}_1(t)-\hat{h}_1(t)|^2+|\tilde{h}_2(t)-\hat{h}_2(t)|^2)dt\nonumber\\
&\leq& C(\Vert \tilde{B}\Vert_{\infty}^2+\Vert \tilde{W}\Vert_{\infty}^2+\Vert \tilde{B}^H\Vert_{\infty}^2+|x-x_0|^2+|y-y_0|^2)T\nonumber\\
&&+C|x-x_0|^2T^{2H}\nonumber\\
&\leq&C(\Vert \tilde{B}\Vert_{\infty}^2+\Vert \tilde{W}\Vert_{\infty}^2+\Vert \tilde{B}^H\Vert_{\infty}^2+|x-x_0|^2+|y-y_0|^2)T^{2H}.
\eeaa
The proof is completed.\end{proof}

\end{lemma}

%
%

\end{document}